\documentclass[DIV=14,letterpaper]{scrartcl}

\usepackage{tikz}
\usetikzlibrary{calc}
\usetikzlibrary{matrix}
\usepackage{amsmath,amssymb,amsfonts,amsthm}
\usepackage{graphicx}
\allowdisplaybreaks
\usepackage{subfigure,multicol,multirow}
\usepackage{makecell}
\usepackage{hyperref} 
\usepackage{diagbox}

\theoremstyle{plain}
\newtheorem{thm}{\hspace{6mm}Theorem}[section]
\newtheorem{pro}{\hspace{6mm}Proposition}[section]
\newtheorem{lem}{\hspace{6mm}Lemma}[section]

\theoremstyle{definition}

\theoremstyle{remark}

\newcommand{\V}[1]{\mathbf{#1}}

\newcommand{\email}[1]{\href{mailto:#1}{#1}}

\title{Consistency enforcement for the iterative solution of  weak Galerkin finite element approximation of Stokes flow}
\author{
Weizhang Huang\thanks{Department of Mathematics, the University of Kansas, 1460 Jayhawk Blvd, Lawrence, KS 66045, USA (\email{whuang@ku.edu}).}
\and
Zhuoran Wang\thanks{Department of Mathematics, the University of Kansas, 1460 Jayhawk Blvd, Lawrence, KS 66045, USA (\email{wangzr@ku.edu}).}
}

\date{} 

\begin{document}

\maketitle

\textbf{Abstract.}
Finite element discretization of Stokes problems can result in singular, inconsistent saddle point linear algebraic systems.
This inconsistency can cause many iterative methods to fail to converge. In this work, we consider the lowest-order
weak Galerkin finite element method to discretize Stokes flow problems and study a consistency enforcement
by modifying the right-hand side of the resulting linear system. It is shown that the modification of the scheme
does not affect the optimal-order convergence of the numerical solution. Moreover, inexact block diagonal
and triangular Schur complement preconditioners and the minimal residual method (MINRES) and
the generalized minimal residual method (GMRES)
are studied for the iterative solution of the modified scheme. Bounds for the eigenvalues and the residual
of MINRES/GMRES are established. Those bounds show that the convergence of MINRES and GMRES is independent
of the viscosity parameter and mesh size. The convergence of the modified scheme and effectiveness of the preconditioners
are verified using numerical examples in two and three dimensions.




\section{Introduction}
\label{SEC:intro}
We consider the lowest-order weak Galerkin (WG) finite element approximation of Stokes flow problems in the form
\begin{equation}
\begin{cases}
\displaystyle
    -\mu \Delta \mathbf{u} + \nabla p  =  \mathbf{f},
      \quad \mbox{in } \; \Omega,
    \\
    \displaystyle
    \nabla \cdot \mathbf{u}  =  0,
     \quad \text{ in } \Omega,
    \\
    \displaystyle
    \mathbf{u}  =  \mathbf{g},
    \quad \mbox{on } \; \partial \Omega,
\end{cases}
\label{Eqn_StokesBVP}
\end{equation}
where
$ \Omega \subset \mathbb{R}^d$ $(d=2,3) $ is a bounded polygonal/polyhedral domain,
$ \mu>0 $ is the fluid kinematic viscosity,
$ \mathbf{u} $ is the fluid velocity,
$ p $ is the fluid pressure,
$ \mathbf{f} $ is a body force,
$\mathbf{g}$ is a non-zero boundary datum of the velocity satisfying the compatibility condition
$ \int_{\partial \Omega}\mathbf{g}\cdot\mathbf{n}=0 $,
and $\mathbf{n}$ is the unit outward normal to
the boundary of the domain. The resulting linear algebraic system is a singular saddle point system
where the singularity reflects the nonuniqueness of the pressure solution.
While it is well known (e.g., see \cite[Remark 6.12]{Elman-2014} and \cite[Section 10.2]{Vorst_2003})
that Krylov subspace methods, such as the minimal residual method (MINRES) and
the generalized minimal residual method (GMRES), work well for consistent singular systems,
the underlying system is nonconsistent in general when $g$ is not identically zero (cf. Section~\ref{SEC:formulation}).
This inconsistency can cause MINRES and GMRES to fail to converge.

The objective of this work is to study a simple consistency enforcement strategy by modifying the right-hand side
of the linear system that results from the WG discretization of (\ref{Eqn_StokesBVP}).
We shall prove that the optimal-order convergence is not affected by the modification.
Moreover, we consider inexact block diagonal and triangular Schur complement preconditioners
for the efficient iterative solution of the singular but consistent saddle point system
resulting from the modified scheme. Bounds for the eigenvalues and the residual of MINRES 
and GMRES for the corresponding preconditioned systems are established.
These bounds show that the convergence of MINRES and GMRES with
the inexact block diagonal and triangular Schur complement preconditioning is independent
of the fluid kinematic viscosity $\mu$ and mesh size.
It is worth mentioning that the preconditioned system associated a block diagonal Schur
complement preconditioner is diagonalizable and this MINRES can be used for system solving
and spectral analysis can be used to analyze the convergence of MINRES.
On the other hand, the preconditioned system with a block triangular Schur complement preconditioner
is not diagonalizable and we need to use GMRES for the iterative solution of the system.
Moreover, the spectral analysis cannot be used directly to analyze the convergence of GMRES.
Instead, it needs to be combined with Lemmas~{A.1} and A.2 of \cite{HuangWang_arxiv_2024}
that provide a bound for the residual of GMRES in terms of the norm of the off-diagonal blocks
in the preconditioned system and the performance of GMRES to the preconditioned Schur complement.

Numerical solution of Stokes flow problems
has continuously gained attention from researchers. Particularly,
a variety of finite element methods have been studied for those problems; e.g., see
\cite{Ainsworth_SINUM_2022} (mixed finite element methods),
\cite{Bevilacqua_SISC_2024,Wang2_CMAME_2021} (virtual element methods),
\cite{Lederer_JSC_2024,TuWangZhang_ETNA_2020} (hybrid discontinuous Galerkin methods),
and \cite{MR3261511,TuWang_CMA_2018,WangYe_Adv_2016} (weak Galerkin (WG) finite element methods).
We use here the lowest-order WG method for the discretization of Stokes flow problems. 
It is known (cf. Lemma~\ref{lem:err1} or \cite{WANG202290})
that the lowest-order WG method, without using stabilization terms, satisfies the inf-sup condition (for stability)
and has the optimal-order convergence. Moreover, the error in the velocity is independent of the error in the pressure
(pressure-robustness) and the error in the pressure is independent of the viscosity $\mu$ ($\mu$-semi-robustness).
On the other hand, little work has been done so far for the efficient iterative solution of the saddle point system
arising from the WG approximation of Stokes problems.
Recently, the authors considered in \cite{HuangWang_arxiv_2024} 
the iterative solution of the lowest-order WG approximation of Stokes problems
through regularization and provided a convergence analysis for MINRES and GMRES with
inexact block diagonal and triangular Schur complement preconditioners, respectively.

The rest of this paper is organized as follows.
In Section~\ref{SEC:formulation}, the weak formulation for Stokes flow 
and its discretization by the lowest-order WG method are described.
Consistency enforcement and related error analysis are discussed in Section~\ref{sec:consistency}.
In Section~\ref{sec:precond}, the block diagonal and block triangular Schur complement preconditioning and convergence of MINRES and GMRES 
for the modified system are studied.
Section~\ref{SEC:numerical} presents two- and three dimensional numerical experiments to verify the theoretical findings and showcase the effectiveness of the preconditioners.
Finally, the conclusions are drawn in Section~\ref{SEC:conclusions}.

\section{Weak Galerkin discretization for Stokes flow}
\label{SEC:formulation}
In this section we describe the lowest-order weak Galerkin approximation of the Stokes flow problem (\ref{Eqn_StokesBVP}).

The weak formulation of (\ref{Eqn_StokesBVP}) is to 
find $ \mathbf{u} \in H^1(\Omega)^d$ and $ p \in L^2(\Omega) $
such that $ \mathbf{u}|_{\partial\Omega} = \mathbf{g} $ (in the weak sense) and
\begin{equation}
\begin{cases}
  \mu (\nabla\mathbf{u}, \nabla\mathbf{v}) - (p, \nabla\cdot\mathbf{v})
= (\mathbf{f}, \mathbf{v}),
   \quad \forall \mathbf{v} \in H^1_0(\Omega)^d,
  \\ 
  -(\nabla\cdot\mathbf{u}, q)
  =  0,
 \quad \forall q \in L^2(\Omega) .
\end{cases}
\label{VarForm}
\end{equation}
Consider a connected quasi-uniform simplicial mesh $\mathcal{T}_h = \{K\}$ on $\Omega$. A mesh is called to be connected
if any pair of its elements is linked by a chain of elements that share an interior facet with each other.
We introduce the following discrete weak function spaces on $\Omega$:
\begin{align}
     \displaystyle
    \mathbf{V}_h
    & = \{ \mathbf{u}_h = \{ \mathbf{u}_h^\circ, \mathbf{u}_h^\partial \}: \;
      \mathbf{u}_h^\circ|_{K} \in P_0(K)^d, \;
      \mathbf{u}_h^\partial|_e \in P_0(e)^d, \;
      \forall K \in \mathcal{T}_h, \; e \in \partial K \},
    \\ 
    \displaystyle
    W_h &= \{ p_h \in L^2(\Omega): \; p_h|_{K} \in P_0(K), \; \forall K \in \mathcal{T}_h\},
\end{align}
where $P_0(K)$ and $P_0(e)$ denote the set of constant polynomials defined on element $K$ and facet $e$, respectively.
Note that $\mathbf{u}_h \in \mathbf{V}_h$ is approximated on both interiors and facets of mesh elements
while $p_h \in W_h$ is approximated only on element interiors.
Then, we can define the lowest-order WG approximation of the Stokes problem \eqref{Eqn_StokesBVP} as:
finding $ \mathbf{u}_h \in \mathbf{V}_h $ and $ p_h \in W_h $
such that $ \mathbf{u}_h^\partial|_{\partial \Omega} = Q_h^{\partial}\mathbf{g} $ and
\begin{equation}
\begin{cases}
    \displaystyle
    \mu \sum_{K\in\mathcal{T}_h} (\nabla_w \mathbf{u}_h,\nabla_w \mathbf{v})_K 
    -\sum_{K \in \mathcal{T}_h}(p_h^{\circ}, \nabla_w\cdot\mathbf{v})_K
    =  \sum_{K \in \mathcal{T}_h} (\mathbf{f}, \mathbf{\Lambda}_h\mathbf{v})_K,
      \quad \forall \mathbf{v} \in \mathbf{V}_h^0,
    \\ 
    \displaystyle
    -\sum_{K \in \mathcal{T}_h}(\nabla_w\cdot\mathbf{u}_h,q^{\circ})_K
    =  0,
   \quad \forall q \in W_h ,
\end{cases}
\label{scheme}
\end{equation}
where $Q_h^{\partial}$ is a $L^2$-projection operator onto $\V{V}_h$ restricted on each facet
and the lifting operator $\mathbf{\Lambda}_h: \V{V}_h \to RT_0(\mathcal{T}_h)$ is defined \cite{Mu.2020,WANG202290} as
\begin{equation}
    \displaystyle
    ( (\mathbf{\Lambda}_h\mathbf{v}) \cdot \mathbf{n}, w )_{e}
      = ( \mathbf{v}^{\partial}\cdot\mathbf{n}, w )_{e},
        \quad \forall w \in P_0(e),\; \forall \V{v} \in \V{V}_h, \; \forall e \subset \partial K .
\label{Eqn_DefLambdah}
\end{equation}
Here, $RT_0(K)$ is the lowest-order Raviart-Thomas space, viz.,
\begin{align*}
    RT_0(K) = (P_0(K))^d + \V{x} \, P_0(K).
\end{align*}
Notice that $\mathbf{\Lambda}_h\mathbf{v}$ depends on $\mathbf{v}^{\partial}$
but not on $\mathbf{v}^{\circ}$. 

The discrete weak gradient and divergence operators in (\ref{scheme}) are defined as follows.
For a scalar function or a component of a vector-valued function, $u_h = (u_h^{\circ},u_h^{\partial})$,
the discrete weak gradient operator $\nabla_w: W_h \rightarrow RT_0(\mathcal{T}_h)$ is defined as follows
\begin{equation}
\label{weak-grad-1}
  (\nabla_w u_h, \mathbf{w})_K
  = (u^\partial_h, \mathbf{w} \cdot \mathbf{n})_{\partial K}
  - ( u^\circ_h , \nabla \cdot \mathbf{w})_K,
  \quad \forall \mathbf{w} \in RT_0(K),\quad \forall K \in \mathcal{T}_h ,
\end{equation}
where $\mathbf{n}$ is the unit outward normal to $\partial K$ and $(\cdot, \cdot)_K$
and $( \cdot, \cdot )_{\partial K}$ are the $L^2$ inner product on $K$ and $\partial K$, respectively.
For a vector-valued function $\mathbf{u}_h$, $\nabla_w \mathbf{u}_h$ is viewed as a matrix with each row representing
the weak gradient of a component.
By choosing $\V{w}$ properly in (\ref{weak-grad-1}) and using the fact that $\nabla_w u_h \in RT_0(K)$, we can obtain (e.g., see \cite{HuangWang_CiCP_2015})
\begin{align}
   & \nabla_w \varphi_K^{\circ} = - C_{K} (\V{x}-\V{x}_{K}) ,
   \label{grad_int}
   \\
   & \nabla_w \varphi_{K,i}^{\partial} = \frac{C_{K}}{d+1} (\V{x} -\V{x}_{K})
+ \frac{|e_{K,i}|}{|K|} \V{n}_{K,i}, \; i = 1, ...,d+1,
\label{grad_face}
\end{align}
where $\varphi_K^{\circ}$ and $\varphi_{K,i}^{\partial}$ denote the basis functions of $P_0(K)$
and $P_0(e_{K,i})$, respectively, $e_{K,i}$ denotes the $i$-th facet of $K$,
$\V{n}_{K,i}$ is the unit outward normal to $e_{K,i}$,
\begin{align*}
C_{K} = \frac{d\; |K|}{\| \V{x} - \V{x}_{K} \|_{K}^2} ,
\quad
\V{x}_{K} = \frac{1}{d+1} \sum_{i=1}^{d+1} \V{x}_{K,i} ,
\end{align*}
and $\V{x}_{K,i}$, $i = 1, ..., d+1$ denote the vertices of $K$.

The discrete weak divergence operator $\nabla_w \cdot: \mathbf{V}_h \to \mathcal{P}_0(\mathcal{T}_h)$ is defined
independently through
\begin{equation}
   (\nabla_w \cdot \mathbf{u}, w )_{K}
  = ( \mathbf{u}^\partial , w \mathbf{n})_{ e }
  - ( \mathbf{u}^\circ , \nabla w)_{K},
  \quad
  \forall w \in P_0(K) .
  \label{wk_div1}
\end{equation}
Note that $\nabla_w \cdot \V{u}|_K \in P_0(K)$. Moreover, by taking $w = 1$ we have
\begin{equation}
(\nabla_w \cdot \V{u}, 1)_{K} = \sum_{i=1}^{d+1} |e_{K,i}| \langle\V{u}\rangle_{e_{K,i}}^T \V{n}_{K,i} ,
\label{wk_div2}
\end{equation}
where $\langle\V{u}\rangle_{e_{K,i}}$ denotes the average of $\V{u}$ on facet $e_{K,i}$
and $|e_{K,i}|$ is the $(d-1)$-dimensional measure of $e_{K,i}$.

The scheme \eqref{scheme} achieves the optimal-order convergence as shown in the following lemma.

\begin{lem}
\label{lem:err1}
Let $ \mathbf{u} \in H^{2}(\Omega)^d $ and $p \in H^1(\Omega)$ be the exact solutions for the Stokes problem (\ref{VarForm}) and let $\mathbf{u}_h\in \mathbf{V}_h$ and $p_h \in W_h$ be numerical solutions for \eqref{scheme}. 
Assume that $ \mathbf{f} \in L^2(\Omega)^d $. 
Then, the followings hold true:
\begin{align}
& \| p - p_h \| \le C h \|\V{f}\|,
\label{err-p}
\\
&   \| \nabla \mathbf{u} - \nabla_w \mathbf{u}_h \|
  \leq C h \|\mathbf{u}\|_2,
\label{err-du}
  \\
  & \| \V{u} - \V{u}_h \| = \| \V{u} - \V{u}_h^{\circ} \| \leq C h \|\mathbf{u}\|_2,
\label{err-u}
\\
&   \| Q_h^\circ\mathbf{u} - \mathbf{u}_h^\circ\|
  \leq C h^{2} \|\mathbf{u}\|_2,
\label{err-u0}
\end{align}
where $\| \cdot \| = \| \cdot \|_{L^2(\Omega)}$, $\|\cdot\|_2 = \|\cdot\|_{H^2(\Omega)}$, 
$C$ is a constant independent of $ h $ and $ \mu $, and $Q_h^\circ$ is a $L^2$-projection operator
for element interiors satisfying $Q_h^\circ\mathbf{u}|_K = \langle \V{u}\rangle_K, \;  \forall K \in \mathcal{T}_h$.
\end{lem}

\begin{proof}
The proof of these error estimates can be found in \cite[Theorem 4.5]{MuYeZhang_SISC_2021} and \cite[Theorem 3]{WANG202290}.
\end{proof}


To cast (\ref{scheme}) into a matrix-vector form,
hereafter we use $\V{v}_h$ interchangeably for the WG approximation of $\V{v}$ in $\V{V}_h$
and the vector formed by its values $(\V{v}_{h,K}^{\circ}, \V{v}_{h,K,i}^{\partial})$
for $i = 1, ..., d+1$ and $K \in \mathcal{T}_h$,
excluding those on $\partial \Omega$.
Similarly, $\V{q}_h$ is used to denote both the WG approximation of $q$ and
the vector formed by $q_{h,K}$ for all $K \in \mathcal{T}_h$.
Then, we can write \eqref{scheme} into 
\begin{equation}
    \begin{bmatrix}
        \mu A & -(B^{\circ})^T \\
       -B^{\circ} & \mathbf{0}
    \end{bmatrix}
    \begin{bmatrix}
        \mathbf{u}_h \\
        \mathbf{p}_h
    \end{bmatrix}
    =
    \begin{bmatrix}
        \mathbf{b}_1 \\
        \mathbf{b}_2
    \end{bmatrix},
    \label{scheme_matrix}
\end{equation}
where the matrices $A$ and $B^{\circ}$ and vectors $\V{b}_1$ and $\V{b}_2$ are defined as
\begin{align}
\mathbf{v}^T A \mathbf{u}_h  & =  \sum_{K\in\mathcal{T}_h} (\nabla_w \mathbf{u}_h,\nabla_w \mathbf{v})_K
    = \sum_{K\in\mathcal{T}_h} (\V{u}_{h,K}^{\circ}\nabla_w \varphi_K^{\circ},  \nabla_w \mathbf{v})_K
    \label{A-1}  \\
   & 
   \displaystyle
    + \sum_{K\in\mathcal{T}_h} \sum^{d+1}_{\substack{i = 1 \\e_{K,i} \notin \partial \Omega}}(\V{u}_{h,K,i}^{\partial}\nabla_w \varphi_{K,i}^{\partial},\nabla_w \mathbf{v})_K,
    \quad \forall \mathbf{u}_h, \mathbf{v} \in \mathbf{V}_h^0,
    \notag \\
    \mathbf{q}^T B^{\circ} \mathbf{u}_h  & =  \sum_{K \in \mathcal{T}_h}  (\nabla_{w}\cdot\mathbf{u}_h,q^{\circ})_K
     = \sum_{K\in\mathcal{T}_h} \sum^{d+1}_{\substack{i = 1 \\e_{K,i} \notin \partial \Omega}}
    |e_{K,i}| q_{K}^{\circ}(\V{u}_{h,K,i}^{\partial})^T \V{n}_{K,i},
    \quad \forall \mathbf{u}_h \in \mathbf{V}_h^0, \quad \forall q \in W_h ,
    \label{B-1} 
    \\
\mathbf{v}^T \V{b}_1 & = \sum_{K \in \mathcal{T}_h} (\mathbf{f}, \mathbf{\Lambda}_h\mathbf{v})_K
    - \mu \sum_{K\in\mathcal{T}_h} \sum_{\substack{i = 1\\ e_{K,i} \in \partial \Omega}}^{d+1}( (Q_{h}^{\partial}\V{g}) \nabla_w \varphi_{K,i}^{\partial},\nabla_w \mathbf{v})_K, \; \forall \mathbf{v} \in \mathbf{V}_h^0,
    \label{b1-1}
    \\
\mathbf{q}^T \V{b}_2 & = \sum_{K\in\mathcal{T}_h} \sum^{d+1}_{\substack{i = 1 \\e_{K,i} \in \partial \Omega}}
    |e_{K,i}| q_{K}^{\circ} (Q_{h}^{\partial}\V{g}|_{e_{K,i}} )^T \V{n}_{K,i},
    \quad \forall q \in W_h .
    \label{b2-1}
\end{align}




Notice that (\ref{scheme_matrix}) is a singular saddle point system where the pressure solution is unique up to a constant.
Moreover, the system is inconsistent in general when $\V{g}$ is not identically zero.
To explain this, from (\ref{b2-1}) we have
\begin{align}
\V{b}_2(K) & = \sum^{d+1}_{\substack{i = 1 \\e_{K,i} \in \partial \Omega}}
    |e_{K,i}| (Q_{h}^{\partial}\V{g}|_{e_{K,i}} )^T \V{n}_{K,i},
    \quad \forall K \in \mathcal{T}_h .
\label{b2-2}
\end{align}
This gives rise to
\begin{align}
\label{b2-3}
\sum_{K \in \mathcal{T}_h} \V{b}_2(K)
& = \sum_{K \in \mathcal{T}_h} \sum^{d+1}_{\substack{i = 1 \\e_{K,i} \in \partial \Omega}}
    |e_{K,i}| (Q_{h}^{\partial}\V{g}|_{e_{K,i}} )^T \V{n}_{K,i}
\\
& = \sum_{e \in \partial \Omega} |e| (Q_{h}^{\partial}\V{g}|_{e} )^T \V{n}_{e}
= \sum_{e \in \partial \Omega} \int_{e} (Q_{h}^{\partial}\V{g}|_{e} )^T \V{n} d S,
\notag 
\end{align}
where we have used the fact that $\partial \Omega$ consists of element facets because $\Omega$ is assumed
to be a polygon or a polyhedron. 
If we choose $Q_{h}^{\partial}\V{g}|_{e} = \langle \V{g} \rangle_e$, $\forall e \in \partial \Omega$, from the compatibility condition
we have
\[
\sum_{K \in \mathcal{T}_h} \V{b}_2(K) = \sum_{e \in \partial \Omega} \int_{e} \langle \V{g} \rangle_e^T \V{n} d S
= \sum_{e \in \partial \Omega} \int_{e} \V{g}^T \V{n} d S = 0 .
\]
Unfortunately, in general the averages $\langle \V{g} \rangle_e$ need to
be approximated numerically. The approximation error can cause a non-zero
$\sum_{K \in \mathcal{T}_h} \V{b}_2(K)$ and thus the inconsistency of the system.
This can also happen if $Q_{h}^{\partial}$ is defined differently, for instance, $Q_{h}^{\partial}\V{g}|_{e}$ is defined as
the value of $\V{g}$ at the barycenter of $e$.
Krylov subspace methods, like MINRES and GMRES, may fail to converge when applied to this inconsistent singular system
although they are known to work well for consistent singular systems, at least when the initial guess is taken as zero;
e.g., see \cite[Remark 6.12]{Elman-2014} and \cite[Section 10.2]{Vorst_2003}.

\section{Consistency enforcement and its effects on the optimal-order convergence}
\label{sec:consistency}

In this section we consider an approach of enforcing the consistency/compatibility condition
by modifying the right-hand term $\V{b}_2$ and study the effects of this modification
on the optimal-order convergence of the numerical solution.

We propose to modify $\V{b}_2$ to enforce the consistency. Specifically, we define
\begin{align}
\tilde{\V{b}}_2(K) =  \V{b}_2(K) - \frac{\alpha_h}{N}, \quad \forall K \in \mathcal{T}_h,
\label{b2-4}
\end{align}
where $N$ is the number of elements in $\mathcal{T}_h$ and
\begin{align*}
    \alpha_h = \sum_{K \in \mathcal{T}_h} \V{b}_2(K) =  \sum_{e \in \partial \Omega} \int_{e} (Q_h^\partial \V{g}|_e)^T \V{n} d S .
\end{align*}
By definition, we have $\sum_{K} \tilde{\V{b}}_2(K) = 0$.
Moreover, from the compatibility condition $\int_{\partial \Omega} \V{g} \cdot \V{n} d S = 0$, we can rewrite $\alpha_h$ as
\begin{align}
    \alpha_h =  \sum_{e \in \partial \Omega} \int_{e} (Q_h^\partial \V{g}|_e - \V{g})^T \V{n} d S 
    = \sum_{e \in \partial \Omega} |e| \left ( Q_h^\partial \V{g}|_e - \langle \V{g}\rangle_e \right )^T \V{n}_e .
    \label{alpha}
\end{align}

The modified scheme reads as
\begin{equation}
    \begin{bmatrix}
        \mu A & -(B^{\circ})^T \\
       -B^{\circ} & \mathbf{0}
    \end{bmatrix}
    \begin{bmatrix}
        \mathbf{u}_h \\
        \mathbf{p}_h
    \end{bmatrix}
    =
    \begin{bmatrix}
        \mathbf{b}_1 \\[0.05in]
        \tilde{\V{b}}_2
    \end{bmatrix} ,
    \label{modified_scheme_matrix}
\end{equation}
where $A$, $B^{\circ}$, and $\V{b}_1$ are given in \eqref{A-1}, \eqref{B-1}, and \eqref{b1-1}, respectively.
Note that \eqref{modified_scheme_matrix} is still singular but consistent.



An immediate question about the modified scheme is how much the accuracy of the numerical solution is affected by the modification.
To answer this, we provide an error analysis for the modified scheme (\ref{modified_scheme_matrix}) in the following.


Let $ \mathbf{u}\in H^{2}(\Omega)^d $ and $ p \in H^1(\Omega) $
be the solutions of the Stokes problem (\ref{Eqn_StokesBVP}) and $ \mathbf{u}_h\in \mathbf{V}_h$ and $p_h\in W_h$
be the numerical solutions of \eqref{modified_scheme_matrix}.
We split the error into the projection and discrete error as 
\begin{equation}
\begin{cases}
   \displaystyle
  \mathbf{u} - \mathbf{u}_h = (\mathbf{u} - Q_h\mathbf{u}) + \mathbf{e}_h,
    \quad \mathbf{e}_h = Q_h\mathbf{u} - \mathbf{u}_h,
  \\
  \displaystyle
  p - p_h = (p - Q_h p) + e_h^p,
   \qquad e_h^p = Q_h p - p_h,
\end{cases}
\label{Eqn_DiscErrs}
\end{equation}
where the $L^2$ projection operator is defined as $Q_h = (Q_h^{\circ}, Q_h^{\partial})$.
The operator $Q_h$ is considered to be componentwise when applied to vector-valued functions.
The projection error $ \mathbf{u} - Q_h\mathbf{u} $ and $ p - Q_h p $
are determined by the approximation capacity of the finite element spaces.
Our primary focus is on $ \mathbf{e}_h $ and $ e_h^p $.

\begin{lem}
\label{lem1}
There hold
\begin{align}
\label{trace1}
&  \|w\|_e^2\leq C(h^{-1}\|w\|_K^2+h\|\nabla w\|_K^2),\quad \forall w \in H^1(K), \quad \forall e \in \partial K, \quad \forall K \in \mathcal{T}_h,
\\
\label{trace2}
&  \|W\|_{K}^2 \leq C h \|W\mathbf{n}\|^2_{\partial K}, \quad \forall W \in RT_0^d, \quad \quad \forall K \in \mathcal{T}_h,
\\ 
& \displaystyle 
  \sum_{K \in \mathcal{T}_h} h^{-1}
    \|\mathbf{v} ^{\partial}-\mathbf{v} ^\circ\|_{\partial K}^2
    \leq C |\hspace{-.02in}|\hspace{-.02in}|\mathbf{v}|\hspace{-.02in}|\hspace{-.02in}|^2,
    \quad \forall \mathbf{v} \in \mathbf{V}_h^0 ,
\label{Eqn_BndDscrpDWG}
\\
&
   \sum_{K \in \mathcal{T}_h}\|  \V{\Lambda}_h \V{v} - \V{v}^{\circ} \|_K \le C(\sum_{K \in \mathcal{T}_h}\sum_{e \in \partial K} h \| \V{v}^{\partial} - \V{v}^{\circ} \|_e^2 )^{1/2}, \quad \forall \mathbf{v} \in \mathbf{V}_h ,
   \label{lifting_property}
\end{align}
where $|\hspace{-.02in}|\hspace{-.02in}|\mathbf{v}|\hspace{-.02in}|\hspace{-.02in}| = \| \nabla_w \V{v} \|$.
\end{lem}

\begin{proof}
These results can be found in \cite[Lemma 2]{WANG202290} for (\ref{trace1}),
\cite{LiuSISC2018} for (\ref{trace2}),
\cite[Lemma 3]{WANG202290} for (\ref{Eqn_BndDscrpDWG}),
\cite[Lemma 4]{WANG202290} for (\ref{lifting_property}).
\end{proof}

\begin{lem}\label{inf_sup}
There exists a constant $ \beta>0 $ independent of $ \mu, h $ such that 
\begin{align}
   \sup_{\mathbf{v}\in \mathbf{V}_h^0, 
    |\hspace{-.02in}|\hspace{-.02in}| \mathbf{v} |\hspace{-.02in}|\hspace{-.02in}| \neq 0} 
  \frac{\V{v}^T(B^{\circ})\V{p}}
      {|\hspace{-.02in}|\hspace{-.02in}| \mathbf{v} |\hspace{-.02in}|\hspace{-.02in}|} 
    \geq \beta \|\V{p}\|, 
  \qquad \forall p \in W_h. 
\end{align}
\end{lem}
\begin{proof}
    Proof can be found in \cite[Lemma 7]{WANG202290}.
\end{proof}

\begin{lem}\label{errorequations}
The error equations for the modified scheme (\ref{modified_scheme_matrix}) read as
\begin{equation}
\begin{cases}
\displaystyle
  \V{v}^T A \V{e}_h - \V{v}^T (B^{\circ})^T \V{e}_h^p
   = \; \mu \, \mathcal{R}_1(\mathbf{u},\mathbf{v}) + \mu \, \mathcal{R}_2(\mathbf{u},\mathbf{v}),
   \quad \forall \mathbf{v} \in \mathbf{V}_h^0,
  \\[0.08in] \displaystyle
  - \V{q}^T B^{\circ} \V{e}_h
   = \; \frac{\alpha_h}{N} \sum_{K \in \mathcal{T}_h} q |_K ,
   \quad \forall q \in W_h,    
  \label{Eqn_SchmII_ErrEqn}
\end{cases}
\end{equation}
where
\begin{equation}
\left\{
\begin{array}{l}
  \displaystyle
  \mathcal{R}_1(\V{u}, \mathbf{v})
  =  \sum_{K \in \mathcal{T}_h}
  \sum_{e \in \partial K}
   ( \mathbf{v}^\partial - \mathbf{v}^\circ, \,
      (\mathcal{Q}_h\nabla\mathbf{u} - \nabla\mathbf{u}) \mathbf{n}_e
   )_{e},
  \\ [0.20in]
  \displaystyle
  \mathcal{R}_2(\V{u},\mathbf{v})
  = \sum_{K \in \mathcal{T}_h}
    (\Delta\mathbf{u}, \, \mathbf{\Lambda}_h\mathbf{v}-\mathbf{v}^\circ)_K,
\end{array}
\right.
\label{Eqn_R1R2}
\end{equation}
and $\mathcal{Q}_h$ is $L^2$ projection from $L^2(\Omega)^{d\times d}$ onto  $RT_0(\mathcal{T}_h)^d$ space.
\end{lem}    
\begin{proof}
Using the discrete errors defined in \eqref{Eqn_DiscErrs}, we substitute
\[
\V{u}_h = Q_h \V{u} -\V{e}_h , \quad p_h = Q_h p - e_h^p 
\]
into the modified scheme \eqref{modified_scheme_matrix} and obtain
\begin{align}
\label{scheme_errII_1}
\begin{cases}
    \displaystyle
&    \mu \sum\limits_{K\in\mathcal{T}_h} (\nabla_w \mathbf{e}_h,\nabla_w \mathbf{v})_K 
    -\sum\limits_{K \in \mathcal{T}_h}(e_h^p, \nabla_w\cdot\mathbf{v})_K
    \displaystyle =  -\sum\limits_{K \in \mathcal{T}_h} (\mathbf{f}, \mathbf{\Lambda}_h\mathbf{v})_K 
    \\
    & \qquad \qquad \qquad \displaystyle +\; \mu \sum_{K\in\mathcal{T}_h} (\nabla_w {Q}_h \V{u},\nabla_w \mathbf{v})_K 
     - \sum\limits_{K \in \mathcal{T}_h}(Q_h p, \nabla_w\cdot\mathbf{v})_K,
      \quad \forall \mathbf{v} \in \mathbf{V}_h^0,
    \\[0.3in]
    \displaystyle
&    -\sum\limits_{K \in \mathcal{T}_h}(\nabla_w\cdot\mathbf{e}_h,q^{\circ})_K
     =  -\frac{\alpha_h}{N} \sum\limits_{K \in \mathcal{T}_h} q |_K 
     - \sum\limits_{K \in \mathcal{T}_h} (\nabla_w \cdot ({Q}_h \V{u}), q)_K,
   \quad \forall q \in W_h .
\end{cases}
\end{align}
Next, we estimate the terms on the right-hand sides. For the first term on the right-hand side of the first equation of
(\ref{scheme_errII_1}), by testing the first equation of \eqref{Eqn_StokesBVP} by $\mathbf{\Lambda}_h\mathbf{v}$ we obtain
\begin{align}
   \sum_{K \in \mathcal{T}_h} (\V{f}, \mathbf{\Lambda}_h\mathbf{v})_K & = \sum_{K \in \mathcal{T}_h} -\mu (\Delta \V{u}, \mathbf{\Lambda}_h\mathbf{v})_K + \sum_{K \in \mathcal{T}_h}  (\nabla p, \mathbf{\Lambda}_h\mathbf{v})_K \notag \\
   & = -\sum_{K \in \mathcal{T}_h} \mu (\Delta \V{u}, \mathbf{\Lambda}_h\mathbf{v} -\V{v}^{\circ})_K - \sum_{K \in \mathcal{T}_h} \mu (\Delta \V{u}, \V{v}^{\circ})_K + \sum_{K \in \mathcal{T}_h} (\nabla p, \mathbf{\Lambda}_h\mathbf{v})_K .
   \label{W_Lift_1}
\end{align}
For the second term on the right-hand side of the above equation,
using the divergence theorem and $\nabla \V{v}^{\circ} = 0$ we have
\begin{align}
   -\sum_{K \in \mathcal{T}_h} \mu (\Delta \V{u}, \V{v}^{\circ})_K & = \sum_{K \in \mathcal{T}_h} \mu (\nabla \V{u}, \nabla \V{v}^{\circ})_K -\sum_{e \in  \partial K} \mu (\nabla \V{u} \cdot \V{n}, \V{v}^{\circ})_e 
     =  \sum_{e \in \partial K} \mu (\nabla \V{u} \cdot \V{n}, \V{v}^{\partial} 
- \V{v}^{\circ})_e.
      \label{W_Lift_2}
\end{align}
For the last term on the right-hand side of (\ref{W_Lift_1}), using the divergence theorem,
the definitions of the lifting and discrete divergence operators, and the fact $\nabla w = 0$ for $w \in P_0(K)$, we have
\begin{align*}
    (\nabla \cdot \Lambda_h(\mathbf{v}), w)_K & = \sum_{e \in \partial K}(\Lambda_h(\mathbf{v}) \cdot \V{n}, w)_e - (\Lambda_h(\mathbf{v}), \nabla w)_K  \notag
    \\
    & =\sum_{e \in \partial K} (\Lambda_h(\mathbf{v}) \cdot \V{n}, w)_e \notag 
    = \sum_{e \in \partial K}(\mathbf{v}\cdot \mathbf{n}, w)_e \notag \\
    & = (\nabla_w \cdot \V{v}, w)_K + (\V{v}, \nabla w)_K 
     = (\nabla_w \cdot \V{v}, w)_K ,
\end{align*}
or
\begin{align}
    \label{lifting_op_2}
    (\nabla \cdot \Lambda_h(\mathbf{v}), w)_K = (\nabla_w \cdot \V{v}, w)_K, \quad \forall w \in P_0(K), \;
    \forall \mathbf{v} \in \mathbf{V}_h .
\end{align}
From this and the divergence theorem, definition of projection $Q_h$, and normal continuity of $\mathbf{\Lambda}_h\mathbf{v}$, 
we obtain
\begin{align}
    \sum_{K \in \mathcal{T}_h}  (\nabla p, \mathbf{\Lambda}_h\mathbf{v})_K & = 
    \sum_{e \in \partial K} (p,\mathbf{\Lambda}_h\mathbf{v} \cdot \V{n})_e - \sum_{K \in \mathcal{T}_h} (p,\nabla \cdot \mathbf{\Lambda}_h\mathbf{v})_K\notag \\
    & = \sum_{e \in \partial K} (p,\mathbf{\Lambda}_h\mathbf{v} \cdot \V{n})_e- \sum_{K \in \mathcal{T}_h} (Q_hp,\nabla \cdot \mathbf{\Lambda}_h\mathbf{v})_K\notag \\
    & = \sum_{e \in \partial K} (p,\mathbf{\Lambda}_h\mathbf{v} \cdot \V{n})_e - \sum_{K \in \mathcal{T}_h} (Q_hp,\nabla_w \cdot \mathbf{v})_K \notag \\
    & = - \sum_{K \in \mathcal{T}_h} (Q_hp,\nabla_w \cdot \mathbf{v})_K .
     \label{W_Lift_3}
\end{align}
Inserting \eqref{W_Lift_2} and \eqref{W_Lift_3} into \eqref{W_Lift_1}, we obtain
\begin{align}
    \sum_{K \in \mathcal{T}_h} (\V{f}, \mathbf{\Lambda}_h\mathbf{v})_K
   & = -\sum_{K \in \mathcal{T}_h} \mu (\Delta \V{u}, \mathbf{\Lambda}_h\mathbf{v} -\V{v}^{\circ})_K
   +\sum_{e \in \partial K} \mu (\nabla \V{u} \cdot \V{n}, \V{v}^{\partial} 
- \V{v}^{\circ})_e
- \sum_{K \in \mathcal{T}_h} (Q_hp,\nabla_w \cdot \mathbf{v})_K .
   \label{W_Lift_4}
\end{align}

For the second term on the right-hand side of the first equation of (\ref{Eqn_SchmII_ErrEqn})
using the WG commuting identity
\[
(\nabla_w {Q}_h \V{u},\V{w})_K = (\mathcal{Q}_h \nabla \V{u},\V{w})_K, \quad \forall \V{w} \in RT_0(K)^{d},
\]
the definition of the discrete weak gradient operator, the divergence theorem, 
and the fact $\nabla \V{v}^{\circ} = 0$, we have
\begin{align}
    \mu \sum_{K\in\mathcal{T}_h} (\nabla_w {Q}_h \V{u},\nabla_w \mathbf{v})_K & = \mu \sum_{K\in\mathcal{T}_h} (\mathcal{Q}_h \nabla \V{u},\nabla_w \mathbf{v})_K \notag \\
    & = \sum_{e \in \partial K} \mu (\mathcal{Q}_h \nabla \V{u} \cdot \V{n}, \V{v}^{\partial})_e -  \sum_{K\in\mathcal{T}_h} \mu (\nabla \cdot (\mathcal{Q}_h \nabla \V{u}), \mathbf{v}^{\circ})_K \notag \\
    & = \sum_{e \in \partial K} \mu (\mathcal{Q}_h \nabla \V{u} \cdot \V{n}, \V{v}^{\partial})_e 
    + \sum_{K\in\mathcal{T}_h}  \mu (\mathcal{Q}_h \nabla \V{u}, \nabla \V{v}^{\circ})_K 
    \notag \\
    & \qquad - \sum_{e \in \partial K} (\mathcal{Q}_h \nabla \V{u} \cdot \V{n}, \V{v}^{\circ})_e \notag \\
    & = \sum_{e \in \partial K} \mu (\mathcal{Q}_h \nabla \V{u} \cdot \V{n}, \V{v}^{\partial} - \mathbf{v}^{\circ})_e  .
    \label{W_Lift_5}
\end{align}
Substituting \eqref{W_Lift_4} and \eqref{W_Lift_5} into \eqref{scheme_errII_1}, we obtain
(\ref{Eqn_SchmII_ErrEqn}).


Now we estimate the second term on the right-hand side of the second equation of \eqref{scheme_errII_1}.
Using the fact $\nabla \cdot \V{u} = 0$ and the commuting identity of WG
\[
(\nabla_w \cdot ({Q}_h \V{u}),w)_K = ({Q}_h (\nabla \cdot \V{u}),w)_K, \quad \forall w \in P_0(K),
\]
we have
    \begin{align*}
        (\nabla_w \cdot (Q_h \V{u}), q)_K = ({Q}_h(\nabla \cdot \V{u}),q)_K = (\nabla \cdot \V{u},q)_K = 0 .
    \end{align*}
Inserting this into \eqref{scheme_errII_1} gives the second error equation in \eqref{Eqn_SchmII_ErrEqn}.
\end{proof}

\begin{lem}
\label{lemma31}
Let $ \mathbf{u} \in H^{2}(\Omega)^d $ and $ p \in H^1(\Omega) $ be the solutions of the Stokes problem and 
$\mathbf{u}_h\in \mathbf{V}_h$ and $p_h\in W_h$ be the numerical solutions of the modified scheme (\ref{modified_scheme_matrix}).
Then,
\begin{align}
  |\hspace{-.02in}|\hspace{-.02in}|  Q_h\mathbf{u}-\mathbf{u}_h |\hspace{-.02in}|\hspace{-.02in}|
  \leq C \Big( h \|\mathbf{u}\|_{2} + |\alpha_h|  \Big),
  \label{pro33}
  \\
  \|Q_hp-p_h\| \leq C \mu \Big( h \|\mathbf{u}\|_{2} + |\alpha_h| \Big),
  \label{pro34}
\end{align}
where $ C>0 $ is a constant independent of $ h $ and $ \mu $.
\end{lem}

\begin{proof}
By the Cauchy-Schwarz inequality and trace inequalities (\eqref{trace1}, \eqref{Eqn_BndDscrpDWG}, and
\eqref{lifting_property} in Lemma~\ref{lem1}), we have
\begin{align}
    | \mathcal{R}_1 (\V{u}, \V{e}_h) |
    &= | \sum_{K \in \mathcal{T}_h}
  \sum_{e \in \partial K}
   (\V{e}_h^{\partial} - \V{e}_h^{\circ}, \,
      (\mathcal{Q}_h\nabla\mathbf{u} - \nabla\mathbf{u}) \mathbf{n}_e
   )_{e} | \nonumber
   \\ 
   & \le
    (\sum_{K \in \mathcal{T}_h} \sum_{e \in \partial K} h^{-1} \| \V{e}_h^{\partial} - \V{e}_h^{\circ} \|_e^2)^{1/2} (\sum_{K \in \mathcal{T}_h} \sum_{e \in \partial K}  h \| (\mathcal{Q}_h\nabla\mathbf{u} - \nabla\mathbf{u}) \V{n}_e\|_e^2)^{1/2} \nonumber
   \\
   & \le C h \| \V{u} \|_2  |\hspace{-.02in}|\hspace{-.02in}|\V{e}_h|\hspace{-.02in}|\hspace{-.02in}| ,
   \label{lem3-R1}
\end{align}
and
\begin{align}
  |\mathcal{R}_2 (\V{u},\V{e}_h) |
    &= |\sum_{K \in \mathcal{T}_h}
   (\Delta \V{u}, \V{\Lambda}_h \V{e}_h - \V{e}_h^{\circ}) |
\le C  \|\V{u}\|_2 (\sum_{K \in \mathcal{T}_h}  \| \V{\Lambda}_h \V{e}_h - \V{e}_h^{\circ} \|^2)^{1/2}
   \notag \\
   & \le C h  \| \V{u} \|_2  |\hspace{-.02in}|\hspace{-.02in}|\V{e}_h|\hspace{-.02in}|\hspace{-.02in}|.
    \label{lem3-R2}
\end{align}
Taking $\V{v} = \V{e}_h$ and $\V{q} = \V{e}_h^p$ in the error equations \eqref{Eqn_SchmII_ErrEqn}, we have
    \begin{align}
    \displaystyle
    \nonumber
        \mu |\hspace{-.02in}|\hspace{-.02in}|\V{e}_h|\hspace{-.02in}|\hspace{-.02in}|^2 
        & = \V{e}_h^T A \V{e}_h 
        = 
        \mu  \mathcal{R}_1(\V{u}, \V{e}_h) + \mu  \mathcal{R}_2(\V{u}, \V{e}_h)
        + \V{e}_h^T (B^{\circ})^T \V{e}_h^p 
        \\ 
       & \le C h \mu \| \V{{u}} \|_2 |\hspace{-.02in}|\hspace{-.02in}|\V{e}_h|\hspace{-.02in}|\hspace{-.02in}| + \frac{|\alpha_h| }{N} \sum_{K \in \mathcal{T}_h} | e_h^p |_K |
       \notag \\
       & \le C h \mu \| \V{{u}} \|_2 |\hspace{-.02in}|\hspace{-.02in}|\V{e}_h|\hspace{-.02in}|\hspace{-.02in}| + C |\alpha_h| \sum_{K \in \mathcal{T}_h} |K| \, | e_h^p |_K |
       \notag \\
       & \le C h \mu \| \V{{u}} \|_2 |\hspace{-.02in}|\hspace{-.02in}|\V{e}_h|\hspace{-.02in}|\hspace{-.02in}| + C |\alpha_h|\, \| e_h^p\| ,
       \label{pro31}
    \end{align}
    where we have used the assumption that the mesh is quasi-uniform. From 
    the inf-sup condition Lemma~\ref{inf_sup}, we get
\begin{align}
\nonumber
    \beta \| \V{e}_h^p \| 
    &\le \sup_{\V{v} \in \V{V}_h^{\circ},\; |\hspace{-.02in}|\hspace{-.02in}|\V{v}|\hspace{-.02in}|\hspace{-.02in}| \neq 0} \frac{\V{v}^T (B^{\circ})^T \V{e}_h^p}{ |\hspace{-.02in}|\hspace{-.02in}|\V{v}|\hspace{-.02in}|\hspace{-.02in}|} \\ \nonumber
    &= \sup_{\V{v} \in \V{V}_h^{\circ},\; |\hspace{-.02in}|\hspace{-.02in}|\V{v}|\hspace{-.02in}|\hspace{-.02in}| \neq 0} \frac{ \V{v}^T A \V{e}_h  - \mu  \mathcal{R}_1(\V{u},\V{v}) -\mu  \mathcal{R}_2(\V{u},\V{v})}{|\hspace{-.02in}|\hspace{-.02in}|\V{v}|\hspace{-.02in}|\hspace{-.02in}|}
    \\ 
    &= \mu |\hspace{-.02in}|\hspace{-.02in}|\V{e}_h|\hspace{-.02in}|\hspace{-.02in}|  + C h \mu \| \V{u} \|_2 .
     \label{pro32}
\end{align}
Combining \eqref{pro31} and \eqref{pro32}, we have
\begin{align}
\nonumber
    \mu |\hspace{-.02in}|\hspace{-.02in}|\V{e}_h|\hspace{-.02in}|\hspace{-.02in}|^2 
    & \le C \mu |\hspace{-.02in}|\hspace{-.02in}|\V{e}_h|\hspace{-.02in}|\hspace{-.02in}|
    (h \| \V{u} \|_2  +  |\alpha_h| )  + C |\alpha_h| h \| \V{u} \|_2 .
\end{align}
Applying the quadratic formula to the above inequality we obtain
\begin{align*}
|\hspace{-.02in}|\hspace{-.02in}|\V{e}_h|\hspace{-.02in}|\hspace{-.02in}|
\le C( h \| \V{u} \|_2 + |\alpha_h|) + C \sqrt{h |\alpha_h|\, \| \V{u} \|_2}
\le C ( h \| \V{u} \|_2 + |\alpha_h|),
\end{align*}
which gives (\ref{pro33}). The inequality (\ref{pro34}) follows from (\ref{pro33}) and \eqref{pro32}.
\end{proof}

\begin{thm}
\label{thm:modified_scheme_err}
Let $ \mathbf{u} \in H^{2}(\Omega)^d $ and $ p \in H^1(\Omega) $ be the solutions of Stokes problem and 
$\mathbf{u}_h\in \mathbf{V}_h$ and $p_h\in W_h$ be the numerical solutions of the modified scheme (\ref{modified_scheme_matrix}).
Assume $ \mathbf{f} \in L^2(\Omega)^d $.
Then,
\begin{align}
    & \| p - p_h \| \le C h \| \V{f} \| + C \mu |\alpha_h|,
    \label{thm:modified_scheme_err-1}
    \\
   &  \| \nabla \V{u} - \nabla_w \V{u}_h \| \le C h \|\V{u} \|_2 +  C |\alpha_h|,
   \label{thm:modified_scheme_err-3}
    \\
    & \|\mathbf{u} - \mathbf{u}_h\| = \|\mathbf{u} - \mathbf{u}_h^\circ\| \leq C h \|\mathbf{u}\|_{2}
    + C |\alpha_h| ,
    \label{thm:modified_scheme_err-2}
    \\
      & 
   \displaystyle \|Q^{\circ}_h \V{u} - \V{u}^{\circ}_h\|  
  \le Ch^2 \| \V{u}\|_2 + C |\alpha_h|,
   \label{thm:modified_scheme_err-2-0}
\end{align}
where $C$ is a constant independent of $h$ and $\mu$.
\end{thm}

\begin{proof}
 
It is known that the solutions of the Stokes problem have the regularity
\[
\mu \| \V{u} \|_2 + \| p\|_1 \le C \| \V{f} \| .
\]
Using this, Lemma~\ref{lemma31}, and the triangle inequality, we obtain (\ref{thm:modified_scheme_err-1}).

Let $ \mathcal{Q}_h $ be a $L^2$-projection operator from $ L^2(\Omega)^{d \times d} $ onto
broken $ RT_0(\mathcal{T}_h)^d $ space.
From the commuting identity of WG, Lemma~\ref{lemma31}, and projection properties, we have
\begin{align*}
    \|\nabla_w \V{u}_h - \nabla \V{u} \| 
    & \le 
    \| \nabla_w \V{u}_h - \mathcal{Q}_h \nabla \V{u} \| 
    +  \| \mathcal{Q}_h \nabla \V{u} - \nabla \V{u}\| 
    \le C(h \| \V{u} \|_2 + |\alpha_h|) ,
\end{align*}
which gives (\ref{thm:modified_scheme_err-3}).

Consider the dual problem 
\begin{equation}
\begin{cases}
  \displaystyle
  -\mu \Delta\mathbf{\Psi}+\nabla\psi =  \mathbf{e}_h^\circ,
    \quad \mbox{in} \;\; \Omega,
  \\[0.05in]
  \displaystyle
  \nabla\cdot\mathbf{\Psi}  = 0,
    \quad \mbox{in} \;\; \Omega,
  \\[0.05in]
  \displaystyle
  \mathbf{\Psi}  = \mathbf{0},
    \quad \mbox{on} \;\; \partial\Omega,
\label{Eqn_DualProb}
\end{cases}
\end{equation}
where $\mathbf{e}_h^\circ = Q^{\circ}_h \V{u} - \V{u}^{\circ}_h$ is the velocity discrete error.
It is known $\mathbf{\Psi} \in H^2(\Omega)^d$ and $\psi \in H^1(\Omega)$.
Taking $\mathbf{e}_h^\circ$ as the test function in the first equation of the dual problem, we have
\begin{align}
    \| \mathbf{e}_h^\circ \|^2 = - \mu \sum_{K \in \mathcal{T}_h} (\Delta\mathbf{\Psi}, \, \mathbf{e}_h^\circ)_{K} + \sum_{{K \in \mathcal{T}_h}}(\nabla\psi, \, \mathbf{e}_h^\circ)_{K}.
\label{thm:modified_scheme_err2-4}
\end{align}
We estimate the terms on the right-hand side separately.

For the second term in \eqref{thm:modified_scheme_err2-4}, using the divergence theorem and the definition of $Q_h$ and discrete weak divergence, we have
\begin{align*}
  \displaystyle
  (\nabla\psi, \mathbf{e}_h^\circ)_K
   &= ( \psi\mathbf{n}, \, \mathbf{e}_h^\circ )_{e}
    - (\psi, \nabla\cdot\mathbf{e}_h^\circ)_K
  \\ 
  \qquad
  \displaystyle
 & = ( \psi\mathbf{n}, \, \mathbf{e}_h^\circ )_{e}
    - (Q_h\psi, \nabla\cdot\mathbf{e}_h^\circ)_K
  \\ 
  \qquad
  \displaystyle
 & = ( \psi\mathbf{n}, \mathbf{e}_h^\circ )_{e}
    - ( (Q_h\psi)\mathbf{n}, \mathbf{e}_h^\circ )_{e}
    + (\nabla(Q_h\psi), \mathbf{e}_h^\circ)_K
  \\
  \qquad
  \displaystyle
  &= ( \psi\mathbf{n}, \mathbf{e}_h^\circ )_{e}
    - ( (Q_h\psi)\mathbf{n}, \mathbf{e}_h^\circ )_{e}
    + ( (Q_h\psi)\mathbf{n}, \mathbf{e}_h^\partial )_{e}
    - (Q_h\psi, \nabla_w\cdot\mathbf{e}_h)_K.
\end{align*}
Summing over $\mathcal{T}_h$ and using
$ \displaystyle
  \sum_{K \in \mathcal{T}_h} \sum_{e \in \partial K}
  ( \psi\mathbf{n}, \, \mathbf{e}_h^\partial )_{e} = 0 $,
we get
\begin{align*}
  \displaystyle
  \sum_{K \in \mathcal{T}_h} (\nabla\psi, \mathbf{e}_h^\circ)_K
  = \sum_{K \in \mathcal{T}_h}\sum_{e \in \partial K}
    ( (Q_h\psi-\psi)\mathbf{n}, \,
    \mathbf{e}_h^\partial - \mathbf{e}_h^\circ )_{e} 
    - \sum_{K \in \mathcal{T}_h}(Q_h\psi, \nabla_w\cdot\mathbf{e}_h)_K.
\end{align*}
From the Cauchy-Schwarz inequality, \eqref{trace1} and \eqref{Eqn_BndDscrpDWG} of 
Lemma~\ref{lem1}, the projection properties, Lemma~\ref{lemma31}, and
\[
\sum_{K \in \mathcal{T}_h}(Q_h\psi, \nabla_w\cdot\mathbf{e}_h)_K
= B^{\circ} (\V{e}_h,Q_h\psi)
= \frac{\alpha_h}{N}\sum_{K \in \mathcal{T}_h} Q_h\psi |_K \le C \|\psi\|_1 |\alpha_h| ,
\]
we have
\begin{align}
    \displaystyle
 \sum_{K \in \mathcal{T}_h} (\nabla\psi, \mathbf{e}_h^\circ)_K
  & = \sum_{K \in \mathcal{T}_h} \sum_{e \in \partial K}
    \langle (Q_h\psi-\psi)\mathbf{n}, \,
    \mathbf{e}_h^\partial - \mathbf{e}_h^\circ \rangle_{e}
- \sum_{K \in \mathcal{T}_h}(Q_h\psi, \nabla_w\cdot\mathbf{e}_h)_K
  \notag \\
  \displaystyle 
  &\le \left( \sum_{K \in \mathcal{T}_h} \sum_{e \in \partial K} h\|Q_h\psi-\psi \|_{e}^2 \right)^{\frac12}
  \left( \sum_{K \in \mathcal{T}_h} \sum_{e \in \partial K} h^{-1} \|\mathbf{e}_h^\partial - \mathbf{e}_h^\circ \|_{e}^2 \right)^{\frac12}
   \notag \\
  \displaystyle 
  &
  + C \|\psi\|_1 |\alpha_h|
  \notag \\
  \qquad
  \displaystyle
  &\le \left( \sum_{K \in \mathcal{T}_h} (\|Q_h\psi-\psi \|^2_{K}+h^2\|\nabla(Q_h\psi-\psi) \|_{K}^2)\right)^{\frac12}|\hspace{-.02in}|\hspace{-.02in}| \mathbf{e}_h |\hspace{-.02in}|\hspace{-.02in}|
  + C \|\psi\|_1 |\alpha_h|
  \notag \\ 
  \qquad
 &   \displaystyle \le Ch \|\psi\|_1
    |\hspace{-.02in}|\hspace{-.02in}| \mathbf{e}_h |\hspace{-.02in}|\hspace{-.02in}|
    +  C \|\psi\|_1 |\alpha_h|
\notag \\
 & \le 
 C  \|\psi\|_1 ( h^2 \|\mathbf{u}\|_2 +  |\alpha_h| ).
\label{thm:modified_scheme_err2-5}
\end{align}

Next, we focus on the first term of \eqref{thm:modified_scheme_err2-4},
$- \mu \sum_{K \in \mathcal{T}_h} (\Delta\mathbf{\Psi}, \, \mathbf{e}_h^\circ)_{K}$.
By the divergence theorem and $\sum_{K \in \mathcal{T}_h} \sum_{e \in \partial K} (\nabla \V{\Psi} \cdot \V{n}, \V{e}_h^{\partial})_e = 0$
(from the continuity of $\V{\Psi}$), there holds
\begin{align}
    -(\Delta \V{\Psi} , \V{e}_h^{\circ}) 
    & = (\nabla \V{\Psi} , \nabla \V{e}_h^{\circ}) - \sum_{K \in \mathcal{T}_h} \sum_{e\in \partial K} (\nabla \V{\Psi}  \cdot \V{n}, \V{e}_h^{\circ})_e \nonumber
    \\
    & = (\nabla \V{\Psi} , \nabla \V{e}_h^{\circ}) - \sum_{K \in \mathcal{T}_h} \sum_{e\in \partial K} (\nabla \V{\Psi}  \cdot \V{n}, \V{e}_h^{\circ} -  \V{e}_h^{\partial})_e.
    \label{thm:modified_scheme_err2-7}
\end{align}

For the first term in \eqref{thm:modified_scheme_err2-7},
by the property of $\mathcal{Q}_h $, divergence theorem, definition of $(\nabla_w)$, and property of $Q_h$, we have
\begin{align}
    (\nabla \V{\Psi} , \nabla \V{e}_h^{\circ}) 
    &= (\mathcal{Q}_h (\nabla \V{\Psi} ), \nabla \V{e}_h^{\circ}) \nonumber
    \\
    & = -(\V{e}_h^{\circ}, \nabla \cdot ((\mathcal{Q}_h (\nabla \V{\Psi})))
    + \sum_{K \in \mathcal{T}_h} \sum_{e\in \partial K}(\V{e}_h^{\circ}, (\mathcal{Q}_h (\nabla \V{\Psi} )) \cdot \V{n})_e \nonumber
    \\
    & = (\mathcal{Q}_h (\nabla \V{\Psi}), \nabla_w \V{e}_h) + \sum_{K \in \mathcal{T}_h} \sum_{e\in \partial K}(\mathcal{Q}_h (\nabla \V{\Psi}) \cdot \V{n}, \V{e}_h^{\circ} - \V{e}_h^{\partial})_e  \nonumber
    \\
    & = (\nabla_w Q_h\V{\Psi}, \nabla_w \V{e}_h) + \sum_{K \in \mathcal{T}_h} \sum_{e\in \partial K}(\mathcal{Q}_h (\nabla \V{\Psi}) \cdot \V{n}, \V{e}_h^{\circ} - \V{e}_h^{\partial})_e 
     \label{thm:modified_scheme_err2-8} .
\end{align}
Plugging \eqref{thm:modified_scheme_err2-8} into \eqref{thm:modified_scheme_err2-7}, we get
\begin{align}
    -\mu (\Delta \V{\Psi}, \V{e}_h^{\circ}) 
    & = \mu (\nabla_w Q_h\V{\Psi}, \nabla_w \V{e}_h) 
\notag \\
& \quad - \sum_{K \in \mathcal{T}_h} \sum_{e\in \partial K}\mu ((\nabla \V{\Psi}  -  \mathcal{Q}_h(\nabla \V{\Psi})) \cdot \V{n}, \V{e}_h^{\circ} - \V{e}_h^{\partial})_e
    \label{thm:modified_scheme_err2-9} .
\end{align}
Combining \eqref{thm:modified_scheme_err2-4}, \eqref{thm:modified_scheme_err2-5}, and \eqref{thm:modified_scheme_err2-9} gives
\begin{align}
\label{thm:modified_scheme_err2-11}
\| \mathbf{e}_h^\circ \|^2
& \le  C  \|\psi\|_1 ( h^2 \|\mathbf{u}\|_2 +  |\alpha_h| ) + \mu(\nabla_w Q_h\V{\Psi}, \nabla_w \V{e}_h) 
\nonumber
\\
& \quad - \sum_{K \in \mathcal{T}_h} \sum_{e\in \partial K}\mu ((\nabla \V{\Psi}  -  \mathcal{Q}_h(\nabla \V{\Psi})) \cdot \V{n}, \V{e}_h^{\circ} - \V{e}_h^{\partial})_e.
\end{align}

Taking $\V{v} = Q_h \V{\Psi}$ in the first equation of error equation \eqref{Eqn_SchmII_ErrEqn} yields
\begin{align*}
    \mu(\nabla_w \V{e}_h, \nabla_w Q_h \V{\Psi}) = 
    \mu \mathcal{R}_1 (\V{u}, Q_h \V{\Psi}) + \mu \mathcal{R}_2 (\V{u}, Q_h \V{\Psi}) + (e_h^p, \nabla_w \cdot Q_h \V{\Psi} ).
\end{align*}
By the property of projection and $\nabla \cdot \V{\Psi} = 0$, we have
\begin{align*}
    (e_h^p, \nabla_w \cdot Q_h \V{\Psi} ) = (e_h^p, Q_h (\nabla \cdot \V{\Psi}) ) = (e_h^p, \nabla \cdot \V{\Psi}) = 0,
\end{align*}
which gives
\begin{align*}
    \mu(\nabla_w \V{e}_h, \nabla_w Q_h \V{\Psi}) = 
    \mu \mathcal{R}_1 (\V{u}, Q_h \V{\Psi}) + \mu \mathcal{R}_2 (\V{u}, Q_h \V{\Psi}) .
\end{align*}
Using this, from \eqref{thm:modified_scheme_err2-11} we obtain
\begin{align}
     \| \mathbf{e}_h^\circ \|^2 
     & \le C  \|\psi\|_1 ( h^2 \|\mathbf{u}\|_2 +  |\alpha_h| )
     +\mu \mathcal{R}_1 (\V{u}, Q_h \V{\Psi}) + \mu \mathcal{R}_2 (\V{u}, Q_h \V{\Psi})
     \notag \\
    & \quad - \sum_{K \in \mathcal{T}_h} \sum_{e\in \partial K}\mu ((\nabla \V{\Psi}  -  \mathcal{Q}_h(\nabla \V{\Psi})) \cdot \V{n}, \V{e}_h^{\circ} - \V{e}_h^{\partial})_e .
    \label{thm:modified_scheme_err2-13}
\end{align}
Next, we estimate the last three terms in the above inequality.

By the trace inequalities \eqref{trace1} and \eqref{Eqn_BndDscrpDWG} and Lemma~\ref{lemma31}, the third term is bounded as
\begin{align}
   & \sum_{K \in \mathcal{T}_h} \sum_{e\in \partial K}\mu ((\nabla \V{\Psi}  -  Q_h\V{\Psi}) \cdot \V{n}, \V{e}_h^{\circ} - \V{e}_h^{\partial})_e \nonumber
    \\
   & \le \mu (\sum_{K \in \mathcal{T}_h} \sum_{e\in \partial K} h \|\nabla \V{\Psi}  -   \mathcal{Q}_h(\nabla \V{\Psi}) \|^2_e)^{1/2} (\sum_{K \in \mathcal{T}_h} \sum_{e\in \partial K} h^{-1} \|\V{e}_h^{\circ} - \V{e}_h^{\partial} \|^2_e)^{1/2} \nonumber
    \\
    & \le C \mu h \| \V{\Psi} \|_2 |\hspace{-.02in}|\hspace{-.02in}| \mathbf{e}_h |\hspace{-.02in}|\hspace{-.02in}| \nonumber
    \\
    & \le C \mu h^2 \| \V{\Psi} \|_2 \| \V{u} \|_2 +  C \mu \| \V{\Psi} \|_2 |\alpha_h| .
     \label{thm:modified_scheme_err2-16}
\end{align}
By the Cauchy-Schwarz inequality and the trace inequality \eqref{trace1}, the second term is bounded as
\begin{align}
    \mu \mathcal{R}_1 (\V{u}, Q_h \V{\V{\Psi}}) 
    &=  \mu \sum_{K \in \mathcal{T}_h}
  \sum_{e \in \partial K}
   (Q_h^{\partial} \V{\Psi} - Q_h^{\circ} \V{\Psi}, \,
      (\mathcal{Q}_h\nabla\mathbf{u} - \nabla\mathbf{u}) \mathbf{n}_e
   )_{e} \nonumber
   \\ 
   & \le
   \mu (\sum_{K \in \mathcal{T}_h} \sum_{e \in \partial K} \| Q_h^{\circ} \V{\Psi} - \V{\Psi} \|_e^2)^{1/2} (\sum_{K \in \mathcal{T}_h} \sum_{e \in \partial K}  \| \mathcal{Q}_h\nabla\mathbf{u} - \nabla\mathbf{u}\|_e^2)^{1/2} \nonumber
   \\
   & \le C \mu h^2 \| \V{u} \|_2 \| \V{\Psi} \|_2 .
   \label{thm:modified_scheme_err2-14}
\end{align}
By the trace inequalities \eqref{trace1} and \eqref{lifting_property},
for the third term we have
\begin{align}
 \mu \mathcal{R}_2 (\V{u}, Q_h \V{\V{\Psi}}) 
    &= \sum_{K \in \mathcal{T}_h}
   \mu(\Delta \V{u}, \V{\Lambda}_h Q_h\V{\Psi} - Q_h^{\circ} \V{\Psi})
\nonumber    \\
    &= \sum_{K \in \mathcal{T}_h}
   \mu( \Delta \V{u},  \V{\Lambda}_h Q_h\V{\Psi} - Q_h^{\circ} \V{\Psi}) \nonumber
   \notag \\ \displaystyle
   & \le C \mu \|\V{u}\|_2 (\sum_{K \in \mathcal{T}_h}\sum_{e \in \partial K} h \| Q_h^{\partial} \V{\Psi} - Q_h^{\circ} \V{\Psi} \|_e^2)^{1/2}
   \notag \\
   & \le C \mu \|\V{u}\|_2 (\sum_{K \in \mathcal{T}_h}\sum_{e \in \partial K} h \|Q_h^{\circ} \V{\Psi}  - \V{\Psi} \|_e^2)^{1/2}
   \notag \\
   & \le C \mu h^2  \| \V{u} \|_2  \| \V{\Psi} \|_2.
    \label{thm:modified_scheme_err2-15}
\end{align}

Combining \eqref{thm:modified_scheme_err2-13}, \eqref{thm:modified_scheme_err2-16}, \eqref{thm:modified_scheme_err2-14}, and \eqref{thm:modified_scheme_err2-15}, 
and using the regularity 
\[
  \mu  \|\mathbf{\Psi}\|_2 + \|\psi\|_1 \le C \|\mathbf{e}_h^\circ\|,
\]
we obtain (\ref{thm:modified_scheme_err-2-0}).
Inequality (\ref{thm:modified_scheme_err-2}) follows
from (\ref{thm:modified_scheme_err-2-0}), the triangle inequality, and approximation properties.
\end{proof}

Comparing the above theorem with Lemma~\ref{lem:err1}, we can see that the effects of the consistency enforcement
(\ref{b2-4}) are reflected by the factor $|\alpha_h|$ in the error bounds
(\ref{thm:modified_scheme_err-1})--(\ref{thm:modified_scheme_err-2-0}).
In particular, if $\alpha_h = 0$, the bounds in Theorem~\ref{thm:modified_scheme_err}
are the same as those in Lemma~\ref{lem:err1}.
Moreover, from (\ref{alpha}) we can see that the magnitude of $\alpha_h$ depends on how closely
$\langle \V{g}\rangle_e$ is approximated by $Q_h^\partial \V{g}|_e$. Such approximation can be made in $\mathcal{O}(h^2)$,
for instance, by choosing $Q_h^\partial \V{g}|_e$ as the value of $\V{g}$ at the barycenter of $e$.
It can also be made at higher order using a Gaussian quadrature rule for computing $\langle \V{g} \rangle_e$.
When the approximation is at second order or higher,
Theorem~\ref{thm:modified_scheme_err} shows that the optimal convergence order of the WG scheme
is not affected by the consistency enforcement (\ref{b2-4}). 


\section{Convergence analysis of MINRES/GMRES with block Schur complement preconditioning}
\label{sec:precond}

In this section we study the MINRES/GMRES
iterative solution for the modified scheme \eqref{modified_scheme_matrix} (that is singular but consistent)
with block diagonal/triangular Schur complement preconditioning.
We establish bounds for the eigenvalues of the preconditioned systems and for the residual of MINRES/GMRES.

We start with rescaling the unknown variables and rewriting \eqref{modified_scheme_matrix} into
\begin{equation} 
    \begin{bmatrix}
        A & -(B^{\circ})^T \\
       -B^{\circ} & 0
    \end{bmatrix}
    \begin{bmatrix}
        \mu \mathbf{u}_h \\
        \mathbf{p}_h
    \end{bmatrix}
    =
    \begin{bmatrix}
        \mathbf{b}_1 \\
        \mu \tilde{\mathbf{b}}_2
    \end{bmatrix},
    \quad \mathcal{A} = \begin{bmatrix}
        A & -(B^{\circ})^T \\
       -B^{\circ} & 0
    \end{bmatrix}.
    \label{modified_scheme_matrix_2}
\end{equation}
The following lemma provide a bound for the Schur complement of the above system.

\begin{lem}
\label{lem:S-bound}
The Schur complement $S = B^{\circ} A^{-1} (B^{\circ})^T$ for (\ref{modified_scheme_matrix_2}) satisfies
\begin{align}
    0 \le S \le d \; M_p^{\circ},
    \label{lem:S-bound-1}
\end{align}
where $M_p^{\circ}$ is the mass matrix of interior pressure, the sign ``$\le$'' is in the sense of negative semi-definite. 
\end{lem}

\begin{proof}
From the definitions of operator $B^{\circ}$ in \eqref{B-1} and the mass matrix  and the fact
$\nabla_w \cdot \V{u} \in P_0(\mathcal{T}_h)$, we have
\begin{align*}
\V{u}^{T}  (B^{\circ})^T (M_p^{\circ})^{-1} B^{\circ} \V{u} = \sum_{K \in \mathcal{T}_h}(\nabla_w \cdot \V{u}, \nabla_w \cdot \V{u})_K .
\end{align*}
Moreover, it can be verified directly that
\begin{align*}
   \sum_{K \in \mathcal{T}_h}(\nabla_w \cdot \mathbf{u_h},\nabla_w \cdot \mathbf{u_h})_K \le
d \sum_{K \in \mathcal{T}_h}(\nabla_w \mathbf{u}_h,\nabla_w \mathbf{u}_h )_K .
\end{align*}
Then, since both $A$ \eqref{scheme_matrix} and $M_p^{\circ}$ are symmetric and positive definite, we have
\begin{align}
    \sup_{\mathbf{p} \neq 0} \frac{\V{p}^{T}  B^{\circ}{A}^{-1}(B^{\circ})^T \V{p}}{\V{p}^T M_p^{\circ} \V{p}} 
    & = \sup_{\mathbf{p} \neq 0}  \frac{\V{p}^{T} (M_p^{\circ})^{-\frac12}  B^{\circ}{A}^{-1}(B^{\circ})^T
    (M_p^{\circ})^{-\frac12} \V{p}}{\V{p}^T \V{p}}
    \notag
    \\
    & = \sup_{\mathbf{u} \neq 0} \frac{\V{u}^{T}  (B^{\circ})^T (M_p^{\circ})^{-1} B^{\circ}
    \V{u}}{\V{u}^T A \V{u}} \le d ,
    \label{AB-1}
\end{align}
which implies $S \le d \; M_p^{\circ}$.
\end{proof}

\begin{lem}
    \label{lem:eigen_SS}
    The eigenvalues of $\hat{S}^{-1} S$, where $\hat{S} = M_p^{\circ}$ and $S = B^{\circ} A^{-1} (B^{\circ})^T$,
    are bounded by $\gamma_1 = 0$ and $\gamma_i \in [\beta^2, d]$ for $i = 2, ..., N$.
\end{lem}

\begin{proof}
Denoting the eigenvalues of $\hat{S}^{-1} S = (M_p^{\circ})^{-1}B^{\circ} A^{-1} (B^{\circ})^T $
by $\gamma_1 = 0 < \gamma_2 \le \cdots \le \gamma_N$.
From (\ref{AB-1}),
we have $\gamma_N \le d$. Moreover, $\gamma_2$ is the smallest positive eigenvalue of $\hat{S}^{-1} S$, which is equal
to the smallest positive eigenvalue of $A^{-1/2} (B^{\circ})^T (M_p^{\circ})^{-1} B^{\circ} A^{-1/2}$,
which is equal to the square of the inf-sup constant $\beta$, i.e., $\gamma_2 = \beta^2$ (cf. Lemma~\ref{inf_sup}).
Thus, 
$\gamma_1 = 0$ and $\gamma_i \in [\beta^2, d]$ for $i = 2, ..., N$.    
\end{proof}

\subsection{Block diagonal Schur complement preconditioning} 
\label{sec:diagonal}

We first consider a block diagonal Schur complement preconditioner.
Based on Lemma~\ref{lem:S-bound}, we take $\hat{S} = M_p^{\circ}$ as an approximation to the Schur complement $S$ and
choose the block diagonal preconditioner as
\begin{align}
    \mathcal{P}_d = \begin{bmatrix}
        A & 0 \\
        0 & M_p^{\circ}
    \end{bmatrix}.
    \label{PrecondPd}
\end{align}
Since $\mathcal{P}_d$ is SPD, the preconditioned system $\mathcal{P}_d^{-1} \mathcal{A}$ is similar to
$\mathcal{P}_d^{-\frac{1}{2}} \mathcal{A} \mathcal{P}_d^{-\frac{1}{2}}$ which can be expressed as
\begin{align}
    \mathcal{P}_d^{-\frac{1}{2}} \mathcal{A} \mathcal{P}_d^{-\frac{1}{2}}&= 
    \begin{bmatrix}
        A^{\frac{1}{2}} & 0 \\
        0 &(M_p^{\circ})^{\frac{1}{2}}
    \end{bmatrix}^{-1}
    \begin{bmatrix}
        A & -(B^{\circ})^T \\
       -B^{\circ} & 0
    \end{bmatrix}
    \begin{bmatrix}
        A^{\frac{1}{2}} & 0 \\
        0 &(M_p^{\circ})^{\frac{1}{2}}
    \end{bmatrix}^{-1} \notag
    \\ 
    & = \begin{bmatrix}
        \mathcal{I} & -A^{-\frac{1}{2}} (B^{\circ})^T(M_p^{\circ})^{-\frac{1}{2}} \\
        -(M_p^{\circ})^{-\frac{1}{2}} B^{\circ}A^{-\frac{1}{2}} & 0
    \end{bmatrix} .
    \label{diag-precond-system}
\end{align}
The symmetry of the preconditioned system indicates that MINRES can be used for its iterative solution, with convergence analyzed through spectral analysis and residual estimates.

\begin{lem}
\label{lem:eigen_bound_diag}
The eigenvalues of $ \mathcal{P}_d^{-1} \mathcal{A} $ are bounded by 
\begin{align}
    \Bigg[ \frac{1-\sqrt{1+4 d}}{2} , \,
    &\frac{1-\sqrt{1+4 \beta^2}}{2}  \Bigg] 
    \cup \Bigg\{ 0 \Bigg\}  
    \cup \Bigg[ \frac{1+\sqrt{1+4 \beta^2}}{2} , \,
    \frac{1+\sqrt{1+4 d}}{2} \Bigg] .
    \label{eigen_bound_diag}
\end{align}
\end{lem}

\begin{proof}
The eigenvalue problem of the preconditioned system $ \mathcal{P}_d^{-1} \mathcal{A} $ can be expressed as
\begin{align}
    \begin{bmatrix}
        A & -(B^{\circ})^T \\
       -B^{\circ} &  0 
    \end{bmatrix}
    \begin{bmatrix}
        \V{u} \\
        \V{p}
    \end{bmatrix} = \lambda
   \begin{bmatrix}
        A & 0 \\
        0 & M_p^{\circ}
    \end{bmatrix} 
        \begin{bmatrix}
        \V{u} \\
        \V{p}
    \end{bmatrix}
    \label{mu0system}
\end{align}
It is clear that $\lambda = 1$ is not an eigenvalue. 
By eliminating $\V{u}$, we obtain
\begin{align}
    \lambda (\lambda - 1) \hat{S} \V{p} = S \V{p} .
    \notag
\end{align}
Denote the eigenvalues of $\hat{S}^{-1} S$ by $\gamma_1 = 0 < \gamma_2 \le ... \le \gamma_N$. Then, we have
\begin{align}
    \lambda^2 - \lambda - \gamma_i = 0, \quad \text{ or } \quad \lambda = \frac{1 \pm \sqrt{1+4 \gamma_i}}{2},
    \qquad i = 2, ..., N.
    \notag
\end{align}
Combining this with the bound of $\gamma_i$ (cf. Lemma~\ref{lem:eigen_SS}), we obtain the bound (\ref{eigen_bound_diag}) for
the eigenvalues of $ \mathcal{P}_d^{-1} \mathcal{A} $.
\end{proof}


\begin{pro}\label{MINRES_conv}
The residual of MINRES applied to the preconditioned system $\mathcal{P}_{d}^{-1/2} \mathcal{A}\mathcal{P}_{d}^{-1/2} $ is bounded by   
\begin{align}
    \frac{\| \V{r}_{2k+1}\| }{\| \V{r}_0 \|} 
    \le
    2
     \Bigg( \frac{\sqrt{d} - \beta}{\sqrt{d} + \beta}\Bigg)^k .
     \label{MINRES_res}
\end{align}
\end{pro}

\begin{proof}
It is known \cite{MINRES-1975} that the residual of MINRES is given by
\[
\| \V{r}_{2k+1}\| = \min\limits_{\substack{p \in \mathbb{P}_{2k+1}\\ p(0) = 1}} \| p (\mathcal{P}_{d}^{-1}\mathcal{A})  \V{r}_0 \|
\le \min\limits_{\substack{p \in \mathbb{P}_{2k+1}\\ p(0) = 1}} \| p (\mathcal{P}_{d}^{-1} \mathcal{A})\| \; \| \V{r}_0 \|,
\]
where $\mathbb{P}_{2k+1}$ is the set of polynomials of degree up to $2k+1$.
Denote the eigenvalues of $\mathcal{P}_{d}^{-1/2} \mathcal{A}\mathcal{P}_{d}^{-1/2} $ by $\lambda_i$, $ i = 2, ..., 2N-1$.
It is known (e.g., see \cite[Section 8.3.4]{Elman-2014} or \cite[Chapter 10]{Vorst_2003}) that the convergence of
MINRES with zero initial guess is not affected by the singular eigenvalues when applied to consistent systems.
Based on the bounds for the eigenvalues of $\mathcal{P}_{d}^{-1}\mathcal{A}$ (cf. Lemma~\ref{lem:eigen_bound_diag})
and \cite[Theorem 6.13]{Elman-2014} (for convergence of MINRES), we have
\begin{align*}
    \frac{\| \V{r}_{2k+1}\| }{\| \V{r}_0 \|} &\le \min\limits_{\substack{p \in \mathbb{P}_{2k}\\ p(0) = 1}} \max\limits_{ i = 1,2,..,2N-1} | p (\lambda_i)| \notag
    \\
   &\le \min\limits_{\substack{p \in \mathbb{P}_{2k}\\ p(0) = 1}} 
    \max_{\lambda \in  [ \frac{1-\sqrt{1+4 d}}{2} , \,
    \frac{1-\sqrt{1+4 \beta^2}}{2}  ] 
    \cup [ \frac{1+\sqrt{1+4 \beta^2}}{2} , \,
    \frac{1+\sqrt{1+4 d}}{2}],} | p(\lambda_i) | \nonumber 
    \\
  &  \le
     2 \Bigg( \frac{\sqrt{d} - \beta}{\sqrt{d}+\beta} \Bigg)^k.
\end{align*}
\end{proof}

Note that  $\| \V{r}_{2k+2}\| \le \| \V{r}_{2k+1}\|$ holds due to the minimization property of MINRES \cite{Elman-2014}.
An important consequence of Proposition~\ref{MINRES_conv} is that the number of MINRES required to reach a prescribed level of the residual, when preconditioned by $\mathcal{P}_d$, is independent of size of the discrete problem and the parameter $\mu$.
Hence, $\mathcal{P}_d$ is an effective preconditioner for \eqref{modified_scheme_matrix_2}.

\subsection{Block triangular Schur complement preconditioning}
\label{sec:triangular}
We now consider block triangular Schur complement preconditioners and the convergence of GMRES.

We start by pointing out that several inexact block triangular Schur complement preconditioners have been studied
in \cite[Appendix A]{HuangWang_arxiv_2024} for general saddle point systems and are shown to perform very similarly.
As such, we focus here only on the following block lower triangular preconditioner,
\begin{equation}
\label{PrecondP}
\mathcal{P}_t =  \begin{bmatrix} A & 0 \\  -B^{\circ} & - M_p^{\circ} \end{bmatrix} .
\end{equation}
This preconditioner is nonsingular. Moreover, it corresponds to the choice $\hat{S} = M_p^{\circ}$ (the pressure mass matrix)
for the approximation of the Schur complement $S = B^{\circ} A^{-1} (B^{\circ})^T$.


\begin{pro}
\label{GMRES_conv}
The residual of GMRES applied to the preconditioned system $\mathcal{P}_{t}^{-1} \mathcal{A}$ is bounded by
\begin{align}
\frac{\| \V{r}_k\|}{\| \V{r}_0\|}
\le 2 \left (1+d+\Big (\frac{d \, \lambda_{\max} (M_p^{\circ})}{\lambda_{\min} (A)} \Big)^{\frac{1}{2}} \right )
\Big(\frac{\sqrt{d} - \beta}{\sqrt{d} + \beta}\Big)^{k-2} .  
\label{GMRES-residual-5}
\end{align}
\end{pro}

\begin{proof}
Applying \cite[Lemma A.1]{HuangWang_arxiv_2024} to (\ref{modified_scheme_matrix_2}) with preconditioner (\ref{PrecondP}),
we obtain the bound for the GMRES residual as
\begin{align}
    \label{GMRES-residual-1}
    \frac{\| \V{r}_k \|}{\| \V{r}_0\|}
\leq \left(1+\| \hat{S}^{-1}S \|+\|A^{-1}(B^{\circ})^T\|\right) \min\limits_{\substack{p \in \mathbb{P}_{k-1}\\ p(0) = 1}} \| p( \hat{S}^{-1} S) \|.
\end{align}
From Lemma~\ref{lem:S-bound}, we have $\| \hat{S}^{-1}S\| \le d$.
In the following, we estimate $\| A^{-1} (B^{\circ})^T \|$ and $\min\limits_{p} \| p(\hat{S}^{-1}S) \|$.

For $\| A^{-1} (B^{\circ})^T \|$, using (\ref{AB-1}) and the fact that $A$ and $M_p^{\circ}$ are SPD,  we have
\begin{align}
    \| A^{-1} (B^{\circ})^T \|^2 
    &= \sup_{\V{p} \neq 0} \frac{\V{p}B^{\circ} A^{-1} A^{-1} (B^{\circ})^T \V{p}}{\V{p}^T  \V{p}} \nonumber
    \\
   & = \sup_{\V{p} \neq 0} \frac{\V{p} (M_p^{\circ})^{\frac{1}{2}} (M_p^{\circ})^{-\frac{1}{2}}B^{\circ} A^{-1}A^{-1} (B^{\circ})^T (M_p^{\circ})^{-\frac{1}{2}} (M_p^{\circ})^{\frac{1}{2}} \V{p}}{\V{p}^T \V{p}}
   \notag \\
   & \le \lambda_{\max} (A^{-1}) \lambda_{\max} (M_p^{\circ})
 \sup_{\V{u} \neq 0} \frac{\V{u}^T B^{\circ} (M_p^{\circ})^{-1} (B^{\circ})^T \V{u}}{\V{u}^T A\V{u}}
 \notag \\
 & \le \frac{d\, \lambda_{\max} (M_p^{\circ})}{\lambda_{\min} (A)} 
    \label{bound2}.
\end{align}

Next, we look into the factor $\min_p \| p(\hat{S}^{-1}S) \|$.
Notice that both $S$ and $\hat{S}$ are symmetric for the current situation. Moreover, $S$ is singular with its null space
given by $\text{Null}((B^{\circ})^T)$. Since the modified scheme is consistent, the corresponding system with $\hat{S}^{-1}S$
is consistent as well. It is known (e.g., see \cite[Section 8.3.4]{Elman-2014} or \cite[Chapter 10]{Vorst_2003}) that the convergence of
GMRES with zero initial guess is not affected by the singular eigenvalues when applied to consistent systems.
Thus, denoting the eigenvalues of $\hat{S}^{-1}S$ by $\gamma_1 = 0 < \gamma_2 \le \cdots \le \gamma_N$, we have
\[
\min\limits_{\substack{p \in \mathbb{P}_{k-1}\\ p(0) = 1}} \| p(\hat{S}_1^{-1}S_1) \|
= \min\limits_{\substack{p \in \mathbb{P}_{k-2}\\ p(0) = 1}} \max_{i=2,..., N} |p(\gamma_i)|
\le \min\limits_{\substack{p \in \mathbb{P}_{k-2}\\ p(0) = 1}} \max_{\gamma \in [\gamma_2,\gamma_N]} |p(\gamma)| .
\]
From Lemma~\ref{lem:eigen_SS} and using shifted Chebyshev polynomials for the above minmax problem
(e.g., see \cite[Pages 50-52]{Greenbaum-1997}), we get
\[
\min\limits_{\substack{p \in \mathbb{P}_{k-2}\\ p(0) = 1}} \max_{\gamma \in [\beta^2,d]} |p(\gamma)|
\le 2 \Big(\frac{\sqrt{d} - \beta}{\sqrt{d} + \beta}\Big)^{k-2} .
\]
Combining the above results, we obtain (\ref{GMRES-residual-5}).
\end{proof}

Recall that $A$ is the stiffness matrix of the WG approximation of the Laplacian operator.
It is known that for quasi-uniform meshes, both $\lambda_{\max}(M_p^{\circ})$ and $\lambda_{\min} (A)$ are in the order
of $1/N$ and $\lambda_{\max}/\lambda_{\min} (A)$ is bounded above by a constant independent of $h$ and $\mu$.
Then, \eqref{GMRES-residual-5} implies that the convergence of GMRES, applied to the modified scheme (\ref{modified_scheme_matrix_2}),
with the block triangular Schur complement conditioner (\ref{PrecondP}), is independent of $h$ and $\mu$.

\section{Numerical experiments}
\label{SEC:numerical}

In this section we present some two- and three-dimensional numerical results to demonstrate the performance of MINRES/GMRES
with the block Schur complement preconditioners for the modified system (\ref{modified_scheme_matrix_2}).
We use MATLAB's function {\em minres} with $tol = 10^{-9}$ for 2D examples and  $tol = 10^{-8}$ for 3D example, with a maximum of 1000 iterations and the zero vector as the initial guess.
Similarly, we use MATLAB's function {\em gmres} with $tol = 10^{-9}$ for 2D examples and  $tol = 10^{-8}$ for 3D examples, $restart = 30$, and the zero vector as the initial guess.
The implementation of the block preconditioners requires the exact inversion of the diagonal blocks. The inversion of the diagonal mass matrix $M_{p}^{\circ}$ is straightforward.
The leading block $A$ represents the WG approximation of the Laplacian operator, and linear systems associated with $A$ are solved using the conjugate gradient method
preconditioned with incomplete Cholesky decomposition. 
The incomplete Cholesky decomposition is carried out using MATLAB's function {\em ichol} with threshold
dropping and the drop tolerance is $10^{-3}$.
Numerical experiments are conducted on triangular and tetrahedral meshes as shown in Fig.~\ref{fig:Mesh}.

\begin{figure}
    \centering
  \subfigure[A triangular mesh] {\includegraphics[width=0.3\linewidth]{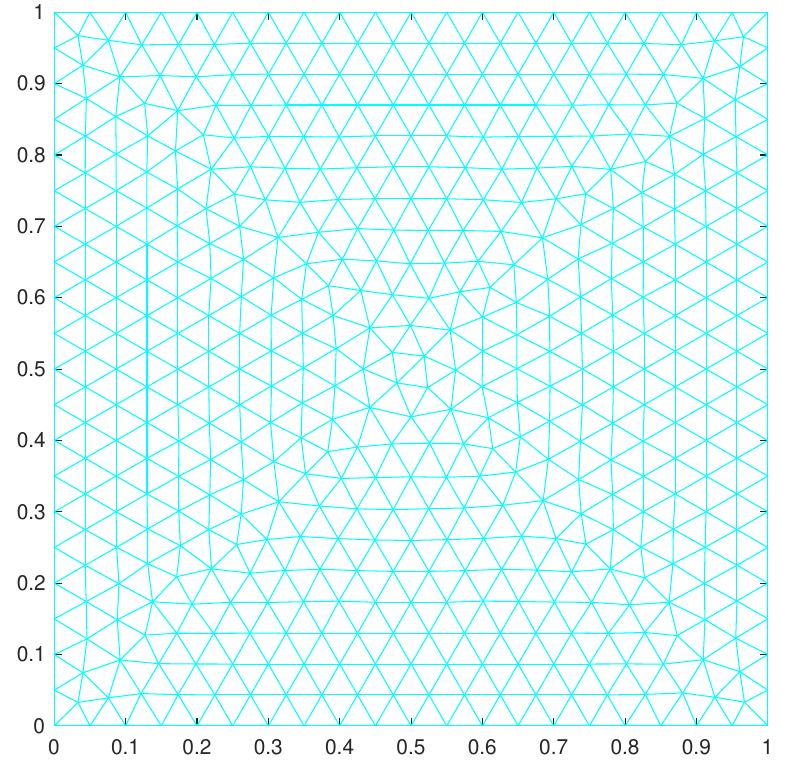}}
  \subfigure[A tetrahedral mesh] {\includegraphics[width=0.3\linewidth]{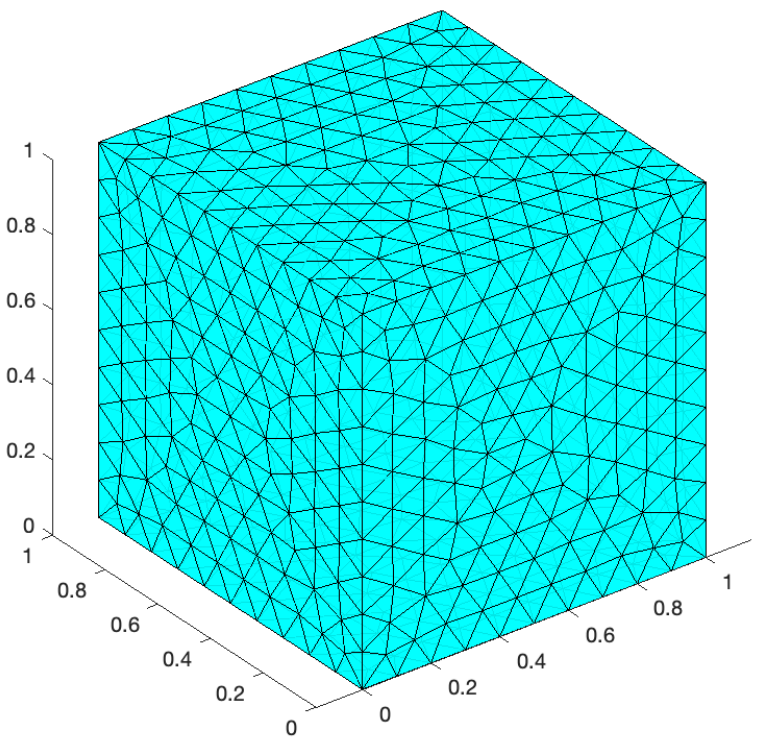}}
    \caption{Examples of meshes used for the computation in two and three dimensions.}
    \label{fig:Mesh}
\end{figure}

\subsection{The two-dimensional example}
This two-dimensional (2D) example is adopted from \cite{Mu.2020} where 
$\Omega = (0,1)^2$, 
\begin{align*}
\V{u} = 
\begin{bmatrix}
    -e^x(y \cos(y)+\sin(y)) \\
    e^x y \sin(y)
\end{bmatrix}, \quad
p = 2 e^x \sin(y),
\quad \V{f} = 
\begin{bmatrix}
    2(1-\mu) e^x  \sin(y) \\
    2(1-\mu) e^x \cos(y)
\end{bmatrix}.
\end{align*}
We take $\mu = 1$ and $10^{-4}$.

To examine the convergence of the modified scheme \eqref{modified_scheme_matrix}, we show the $L^2$ error of the velocity
in Table~\ref{P1conv_2d}. Theoretically, this error has the optimal first-order convergence as stated in Theorem~\ref{thm:modified_scheme_err}.
In Table~\ref{P1conv_2d}, the $L^2$ error of velocity for both $\mu = 1$ and $10^{-4}$ shows the first-order convergence rate.
Interestingly, the error is almost the same for $\mu = 1$ and $10^{-4}$.
\begin{table}[h]
    \centering
        \caption{The 2D Example. The $L^2$ error of the velocity for the modified scheme (\ref{modified_scheme_matrix_2})
        (consistency enforcement) with $\mu = 1$ and $\mu = 10^{-4}$.}
    \begin{tabular}{|c||c|c||c|c|}
        \hline
        & \multicolumn{2}{c||}{$\mu = 1$} & \multicolumn{2}{c|}{$\mu = 10^{-4}$} \\ \hline
        $N$ & $\| \V{u} - \V{u}_h \|$ & conv. rate & $\| \V{u} - \V{u}_h \|$  &conv. rate  \\ \hline
       232  & 9.428879e-02 & -- & 9.428878e-02 & --\\ \hline
       918  &4.719342e-02  & 1.006 &4.719342e-02 &  1.006\\ \hline
       3680  & 2.352019e-02  & 1.003 & 2.352019e-02 & 1.003 \\ \hline  
      14728  & 1.174930e-02 & 1.000 & 1.174930e-02 & 1.000 \\ \hline  
    58608  &  5.893651e-03  & 0.999 & 5.893651e-03  &  0.999\\ \hline
    \end{tabular}
    \label{P1conv_2d}
\end{table}

Recall from Proposition~\ref{MINRES_conv} that the residual of MINRES with the preconditioner $\mathcal{P}_d$
is independent of $\mu$ and $h$. Table \ref{P1steps-2d-minres} shows a consistent number
of MINRES iterations as the mesh is refined. Furthermore, the iteration remains stable for small values of $\mu$.
Similar observations can be made for GMRES from Table~\ref{P1steps-2d}.

The good performance of the block Schur complement preconditioners $\mathcal{P}_d$ and $\mathcal{P}_t$
can be explained from the  number of MINRES/GMRES iterations shown in Tables~\ref{P1steps-2d-minres} and \ref{P1steps-2d}.
The number of iterations required to reach convergence
remains relatively small and consistent with mesh refinement and various values of $\mu$. 
If without preconditioning, on the other hand, MINRES and GMRES would take more than 30,000 iterations to reach convergence,
which is a significant difference compared to the solvers with preconditioning.
We also observe that the number of MINRES iterations is almost double that of GMRES, as shown by comparing Tables~\ref{P1steps-2d-minres} and \ref{P1steps-2d}. This observation is consistent with the estimates in Proposition~\ref{MINRES_conv} and Proposition~\ref{GMRES_conv}.




\begin{table}[tbh!]
    \centering
    \caption{The 2D Example. The number of MINRES iterations required to reach convergence for the preconditioned system $\mathcal{P}_d^{-1} \mathcal{A}$ with block diagonal Schur complement preconditioning.}
    \label{2d-cv-minres}
    \begin{tabular}{|c|c|c|c|c|c|}
    \hline
\diagbox{$\mu$}{$N$} & 232 & 918 & 3680 & 14728 & 58608 \\
    \hline
         $1$ & 43 & 47 & 49 & 49 & 47 \\
   \hline
         $10^{-4}$ & 42 & 48 & 54 & 58 & 60 \\ 
    \hline
    \end{tabular}
    \label{P1steps-2d-minres}
\end{table}



\begin{table}[tbh!]
    \centering
    \caption{The 2D Example. The number of GMRES iterations required to reach convergence for the preconditioned system $\mathcal{P}_t^{-1} \mathcal{A}$ with block triangular Schur complement preconditioning.}
    \label{2d-cv}
    \begin{tabular}{|c|c|c|c|c|c|}
    \hline
\diagbox{$\mu$}{$N$} & 232 & 918 & 3680 & 14728 & 58608 \\
    \hline
         $1$ & 21 & 23 & 24 & 25 & 25 \\
   \hline
         $10^{-4}$ & 23 & 25 & 27 & 27 & 27 \\ 
    \hline
    \end{tabular}
    \label{P1steps-2d}
\end{table}




\subsection{The three-dimensional example}
This three-dimensional (3D) example is adopted from \textit{deal.II} \cite{dealii} \texttt{step-56}. 
Here, $\Omega = (0,1)^3$ and 
\begin{align*}
& \V{u} = 
\begin{bmatrix}
    2 \sin(\pi x) \\
    -\pi y \cos(\pi x) \\
    -\pi z \cos(\pi x)
\end{bmatrix}, \quad
p = \sin(\pi x) \cos(\pi y) \sin(\pi z),
\\
& \V{f} = 
\begin{bmatrix}
    2 \mu \pi^2 \sin(\pi x) + \pi \cos(\pi x) \cos(\pi y) \sin(\pi z) \\
    -\mu \pi^3 y \cos(\pi x) - \pi \sin(\pi y) \sin(\pi x) \sin(\pi z)\\
    -\mu \pi^3 z \cos(\pi x) + \pi \sin(\pi x) \cos(\pi y) \cos(\pi z)
\end{bmatrix}.
\end{align*}

The number of MINRES/GMRES iterations for
the preconditioned systems for $\mu = 1$ and $10^{-4}$ is listed in Tables~\ref{P1-3d-cv-minres} and
\ref{P1-3d-cv}, respectively.
The number of iterations remains relatively small and consistent with various in $\mu$ and $h$,
which is consistent with the fact that the bounds of the residual in Propositions~\ref{MINRES_conv} and \ref{GMRES_conv}
are independent of $\mu$ and $h$.
Once again, MINRES/GMRES without preconditioning would require more than 30,000 iterations to reach convergence. 

\begin{table}[tbh!]
    \centering
    \caption{The 3D Example. The number of MINRES iterations required to reach convergence for the preconditioned system $\mathcal{P}_d^{-1} \mathcal{A}$ with block diagonal Schur complement preconditioning.}
    \label{P1-3d-cv-minres}
    \begin{tabular}{|c|c|c|c|c|c|}
    \hline
\diagbox{$\mu$}{$N$} & 4046 & 7915 & 32724 & 112078 & 266555\\
   \hline
         $1$ & 59&59  &63  & 67 & 67 \\ 
    \hline
         $10^{-4}$ & 62 &62  & 70 &  76 &78  \\
    \hline
    \end{tabular}
\end{table}



\begin{table}[tbh!]
    \centering
    \caption{The 3D Example. The number of GMRES iterations required to reach convergence for the preconditioned system $\mathcal{P}_t^{-1} \mathcal{A}$  with block triangular Schur complement preconditioning.}
    \label{P1-3d-cv}
    \begin{tabular}{|c|c|c|c|c|c|}
    \hline
\diagbox{$\mu$}{$N$} & 4046 & 7915 & 32724 & 112078 & 266555\\
    \hline
         $1$ &30  & 30 & 31 &  32 & 33 \\
   \hline
         $10^{-4}$ & 34& 35 & 37 & 38 & 38 \\ 
    \hline
    \end{tabular}
\end{table}

\section{Conclusions}
\label{SEC:conclusions}
In the previous sections we have studied the iterative solution of the singular saddle point system arising from
the lowest-order weak Galerkin discretization of Stokes flow problems.
In general the system is inconsistent when the boundary datum is not identically zero.
This inconsistency can cause iterative methods such as MINRES and GMRES to fail to converge.
We have proposed a simple strategy to enforce the consistency by modifying the right-hand side of the second equation
(cf. \eqref{b2-4}). We have shown in Theorem~\ref{thm:modified_scheme_err} that
the optimal-order convergence of the numerical solutions is not affected by the modification.

We have considered block diagonal and triangular Schur complement preconditioning
for the iterative solution of the modified system.
In Section~\ref{sec:diagonal}, we have studied the block diagonal Schur complement preconditioner
\eqref{PrecondPd} and established bounds for the eigenvalues of the preconditioned system (see Lemma~\ref{lem:eigen_bound_diag})
as well as  for the residual of MINRES applied to the preconditioned system (cf. Proposition~\ref{MINRES_conv}). 
In Section~\ref{sec:triangular}, we have studied the block triangular Schur complement preconditioner (\ref{PrecondP})
and derived the asymptotic error bound in Proposition~\ref{GMRES_conv} for GMRES applied to the preconditioned system.
These bounds show that both the convergence of MINRES and GMRES for the corresponding preconditioned systems
is independent of $h$ and $\mu$. The optimal-order convergence of the modified scheme \eqref{modified_scheme_matrix_2}
and effectiveness of the block diagonal
and triangular Schur complement preconditioners have been verified by two- and three-dimensional numerical examples
in Section~\ref{SEC:numerical}.

\section*{Acknowledgments}

W.~Huang was supported in part by the Air Force Office of Scientific Research (AFOSR) grant FA9550-23-1-0571
and the Simons Foundation grant MPS-TSM-00002397.


\bibliographystyle{abbrv}

\end{document}